\documentclass[preprint,12pt]{elsarticle}

\usepackage[all]{xy}
\usepackage{amssymb}
\usepackage{amsmath}
\usepackage{graphicx}
\usepackage{color}
\usepackage{amsthm}

\newtheorem{thm}[subsection]{Theorem}
\newtheorem{prop}[subsection]{Proposition}
\newtheorem{lemma}[subsection]{Lemma}
\newtheorem{cor}[subsection]{Corollary}

\newenvironment{proo}{\begin{trivlist} \item{\emph{Proof.}}}
  {\hfill $\square$ \end{trivlist}}

\newcommand{\draftnote}[1]{}

\def\u{\underline }

\def\ss{\sigma}

\def\t{\otimes}

\def\permAs{\mathrm{permAs}}
\def\permAsDer{\mathrm{permAsDer}}
\def\permMag{\mathrm{permMag}}

\def\NN{{\mathbb{N}}}

\def\RR{{\mathbb{R}}}

\def\RR{{\mathbb{R}}}
\def\TT{{\mathbb{T}}}

\def\PP{{\mathcal{P}}}

\def\QQQ{{\mathcal{Q}}}

\def\PPP{{\mathbb{P}}}

\def\Id{\mathrm{Id }}

\def\End{\mathrm{End}}
\def\sgn{\mathrm{sgn}}

\def\epi{\twoheadrightarrow}

\def\End{\mathop{\rm End }}

\def\dim{\mathop{\rm dim }}
\def\id{\mathrm{ id }}

\def\Mod{\textrm{-Mod}}

\def\KK{\mathbb{K}}

\def\Sy{\mathbb{S}}

\def\vert{\texttt{vert}}

\def\arbreA{\vcenter{\xymatrix@R=3pt@C=3pt{
&& \\
&*{}\ar@{-}[ur] \ar@{-}[ul] \ar@{-}[d]     &\\
&&
}}}

\def\arbreAsh{\vcenter{\xymatrix@R=3pt@C=3pt{
0&&1\\
&& \\
&*{}\ar@{-}[ur] \ar@{-}[ul] \ar@{-}[d]     &\\
&&
}}}

\def\arbreAv{\vcenter{\xymatrix@R=3pt@C=3pt{
&v& \\
&*{}\ar@{-}[ur] \ar@{-}[ul] \ar@{-}[d]     &\\
&&
}}}

\def\arbreBA{\vcenter{\xymatrix@R=2pt@C=2pt{
&&&&\\
&&&*{}\ar@{-}[ul] & \\
&&*{}\ar@{-}[uurr] \ar@{-}[uull] \ar@{-}[d]     &&\\
&&&&
}}}

\def\arbreBAsh{\vcenter{\xymatrix@R=2pt@C=2pt{
0&&1&&2\\
&&&&\\
&&&*{}\ar@{-}[ul] & \\
&&*{}\ar@{-}[uurr] \ar@{-}[uull] \ar@{-}[d]     &&\\
&&&&
}}}

\def\arbreBAdecorated{\vcenter{\xymatrix@R=8pt@C=8pt{
&&&&\\
&&&\nu \ar@{-}[ul] \ar@{-}[ur] & \\
&&\mu\ar@{-}[ur] \ar@{-}[uull] \ar@{-}[d]     &&\\
&&&&
}}}

\def\arbreBAv{\vcenter{\xymatrix@R=2pt@C=2pt{
&v_{1}&&v_{2}&\\
&&&*{}\ar@{-}[ul] & \\
&&*{}\ar@{-}[uurr] \ar@{-}[uull] \ar@{-}[d]     &&\\
&&&&
}}}

\def\arbreAB{\vcenter{\xymatrix@R=2pt@C=2pt{
&&&&\\
&*{}\ar@{-}[ur] &&& \\
&&*{}\ar@{-}[uurr] \ar@{-}[uull] \ar@{-}[d]     &&\\
&&&&
}}}

\def\arbreABsh{\vcenter{\xymatrix@R=2pt@C=2pt{
0&&1&&2\\
&&&&\\
&*{}\ar@{-}[ur] &&& \\
&&*{}\ar@{-}[uurr] \ar@{-}[uull] \ar@{-}[d]     &&\\
&&&&
}}}

\def\arbreABshdeux{\vcenter{\xymatrix@R=2pt@C=2pt{
0&&2&&1\\
&&&&\\
&*{}\ar@{-}[ur] &&& \\
&&*{}\ar@{-}[uurr] \ar@{-}[uull] \ar@{-}[d]     &&\\
&&&&
}}}

\def\arbreABdecorated{\vcenter{\xymatrix@R=8pt@C=8pt{
&&&&\\
&\nu\ar@{-}[ur] \ar@{-}[ul] &&& \\
&&\mu\ar@{-}[uurr] \ar@{-}[ul] \ar@{-}[d]     &&\\
&&&&
}}}

\def\arbreABv{\vcenter{\xymatrix@R=2pt@C=2pt{
&v_{1}&&v_{2}&\\
&*{}\ar@{-}[ur] &&& \\
&&*{}\ar@{-}[uurr] \ar@{-}[uull] \ar@{-}[d]     &&\\
&&&&
}}}

\def\arbreBB{\vcenter{\xymatrix@R=2pt@C=2pt{
&&*{}&&\\
&&&& \\
&&*{}\ar@{-}[uurr] \ar@{-}[uull] \ar@{-}[d] \ar@{-}[uu]     &&\\
&&&&
}}}

\def\arbreABC{\vcenter{\xymatrix@R=1pt@C=1pt{
&&&&&&\\
&*{}\ar@{-}[ur] &&&&& \\
&&*{}\ar@{-}[uurr] &&&&\\
&&&*{}\ar@{-}[uuurrr] \ar@{-}[uuulll] \ar@{-}[d] &&&\\
&&&&&&
}}}

\def\arbreABCsh{\vcenter{\xymatrix@R=1pt@C=1pt{
0&&1&&2&&3\\
&&&&&&\\
&*{}\ar@{-}[ur] &&&&& \\
&&*{}\ar@{-}[uurr] &&&&\\
&&&*{}\ar@{-}[uuurrr] \ar@{-}[uuulll] \ar@{-}[d] &&&\\
&&&&&&
}}}

\def\arbreABCshdeux{\vcenter{\xymatrix@R=1pt@C=1pt{
0&&1&&3&&2\\
&&&&&&\\
&*{}\ar@{-}[ur] &&&&& \\
&&*{}\ar@{-}[uurr] &&&&\\
&&&*{}\ar@{-}[uuurrr] \ar@{-}[uuulll] \ar@{-}[d] &&&\\
&&&&&&
}}}

\def\arbreBAC{\vcenter{\xymatrix@R=1pt@C=1pt{
&&&&&&\\
&&&*{}\ar@{-}[ul] &&& \\
&&*{}\ar@{-}[uurr] &&&&\\
&&&*{}\ar@{-}[uuurrr] \ar@{-}[uuulll] \ar@{-}[d] &&&\\
&&&&&&
}}}

\def\arbreACA{\vcenter{\xymatrix@R=1pt@C=1pt{
&&&&&&\\
&*{}\ar@{-}[ur] &&&&*{}\ar@{-}[ul] & \\
&&&&&&\\
&&&*{}\ar@{-}[uuurrr] \ar@{-}[uuulll] \ar@{-}[d] &&&\\
&&&&&&
}}}

\def\arbreAAC{\vcenter{\xymatrix@R=1pt@C=1pt{
&&&&&&\\
&&&&&& \\
&&*{}\ar@{-}[uu]\ar@{-}[uurr] &&&&\\
&&&*{}\ar@{-}[uuurrr] \ar@{-}[uuulll] \ar@{-}[d] &&&\\
&&&&&&
}}}

\def\arbreAACsh{\vcenter{\xymatrix@R=1pt@C=1pt{
0&&1&&2&&3\\
&&&&&&\\
&&&&&& \\
&&*{}\ar@{-}[uu]\ar@{-}[uurr] &&&&\\
&&&*{}\ar@{-}[uuurrr] \ar@{-}[uuulll] \ar@{-}[d] &&&\\
&&&&&&
}}}

\def\arbreAACshdeux{\vcenter{\xymatrix@R=1pt@C=1pt{
0&&1&&3&&2\\
&&&&&&\\
&&&&&& \\
&&*{}\ar@{-}[uu]\ar@{-}[uurr] &&&&\\
&&&*{}\ar@{-}[uuurrr] \ar@{-}[uuulll] \ar@{-}[d] &&&\\
&&&&&&
}}}

\def\arbreACAv{\vcenter{\xymatrix@R=1pt@C=1pt{
&v_{1}&&v_{2}&&v_{3}&\\
&*{}\ar@{-}[ur] &&&&*{}\ar@{-}[ul] & \\
&&&&&&\\
&&&*{}\ar@{-}[uuurrr] \ar@{-}[uuulll] \ar@{-}[d] &&&\\
&&&&&&
}}}

\def\arbreACB{\vcenter{\xymatrix@R=1pt@C=1pt{
&&&&&&\\
&*{}\ar@{-}[ur] &&&&& \\
&&&&*{}\ar@{-}[uull] &&\\
&&&*{}\ar@{-}[uuurrr] \ar@{-}[uuulll] \ar@{-}[d] &&&\\
&&&&&&
}}}

\def\arbreBCA{\vcenter{\xymatrix@R=1pt@C=1pt{
&&&&&&\\
&&&&&*{}\ar@{-}[ul] & \\
&&*{}\ar@{-}[uurr] &&&&\\
&&&*{}\ar@{-}[uuurrr] \ar@{-}[uuulll] \ar@{-}[d] &&&\\
&&&&&&
}}}

\def\arbreCAB{\vcenter{\xymatrix@R=1pt@C=1pt{
&&&&&&\\
&&&*{}\ar@{-}[ur] &&& \\
&&&&*{}\ar@{-}[uull] &&\\
&&&*{}\ar@{-}[uuurrr] \ar@{-}[uuulll] \ar@{-}[d] &&&\\
&&&&&&
}}}

\def\arbreCBA{\vcenter{\xymatrix@R=1pt@C=1pt{
&&&&&&\\
&&&&&*{}\ar@{-}[ul] & \\
&&&&*{}\ar@{-}[uull] &&\\
&&&*{}\ar@{-}[uuurrr] \ar@{-}[uuulll] \ar@{-}[d] &&&\\
&&&&&&
}}}

\def\dessinA{ \begin{picture}(250,100)
\put(0,0){\scriptsize {$\times$}}
\put(20,0){\scriptsize {$\times$}}
\put(78,0){\scriptsize {$\times$}}
\put(98,0){\scriptsize {$\times$}}
\put(2,5){\line(0,1){23}}
\put(22,5){\line(0,1){23}}
\put(2,5){\line(2,1){38}}
\put(80,5){\line(1,1){18}}
\put(80,23){\line(1,-1){18}}
\put(0,25){\scriptsize {$\bullet$}}
\put(20,25){\scriptsize {$\bullet$}}
\put(40,25){\scriptsize {$\bullet$}}
\put(78,25){\scriptsize {$\bullet$}}
\put(98,25){\scriptsize {$\bullet$}}
\put(0,40){\line(1,0){100}}
\put(8,45){\scriptsize {$\times$}}
\put(63,45){\scriptsize {$\times$}}
\put(11,50){\line(0,1){25}}
\put(13,50){\line(3,2){35}}
\put(15,50){\line(5,2){55}}
\put(68,50){\line(1,1){22}}
\put(34,72){\line(4,-3){30}}
\put(8,72){\scriptsize {$\bullet$}}
\put(30,72){\scriptsize {$\bullet$}}
\put(48,72){\scriptsize {$\bullet$}}
\put(70,72){\scriptsize {$\bullet$}}
\put(90,72){\scriptsize {$\bullet$}}
\put(120,40){\scriptsize {$\mapsto$}}
\put(150,16){\scriptsize {$\times$}}
\put(170,16){\scriptsize {$\times$}}
\put(206,16){\scriptsize {$\times$}}
\put(230,16){\scriptsize {$\times$}}
\put(152,20){\line(0,1){24}}
\put(154,20){\line(3,2){37}}
\put(173,20){\line(0,1){24}}
\put(210,20){\line(1,1){25}}
\put(232,20){\line(-1,1){25}}
\put(150,45){\scriptsize {$\bullet$}}
\put(170,45){\scriptsize {$\bullet$}}
\put(190,45){\scriptsize {$\bullet$}}
\put(204,45){\scriptsize {$\bullet$}}
\put(233,45){\scriptsize {$\bullet$}}
\put(149,50){\scriptsize {$\times$}}
\put(169,50){\scriptsize {$\times$}}
\put(189,50){\scriptsize {$\times$}}
\put(203,50){\scriptsize {$\times$}}
\put(232,50){\scriptsize {$\times$}}
\put(152,54){\line(0,1){24}}
\put(173,54){\line(2,3){18}}
\put(192,54){\line(2,3){18}}
\put(205,54){\line(-4,3){33}}
\put(236,54){\line(0,1){24}}
\put(150,78){\scriptsize {$\bullet$}}
\put(170,78){\scriptsize {$\bullet$}}
\put(189,78){\scriptsize {$\bullet$}}
\put(206,78){\scriptsize {$\bullet$}}
\put(233,78){\scriptsize {$\bullet$}}
\end{picture}}

\def\Kzero{\xymatrix@R=4pt@C=4pt{
\\
\\
\\
{\bullet}\\
}}

\def\KunA{\xymatrix@R=4pt@C=4pt{
&&\\
&&\\
&&\\
*{}\ar@{-}[rr]&&*{}\\
}}

\def\PdeuxA{\xymatrix@R=4pt@C=4pt{
&&&&&\\
&&*{}\ar@{-}[dll]\ar@{-}[drr]  &&& \\
*{}\ar@{-}[dd] &&&&*{}\ar@{-}[dd] &\\
&&&&& \\
*{}\ar@{-}[drr] &&&&*{}\ar@{-}[dll] & \\
&&*{}&&& \\
}}

\def\permutohedronLemma{\xymatrix@R=10pt@C=10pt{
&&&[123]&&&\\
&&&*{}\ar[dll]\ar[drr]  &&& \\
[213]&*{}\ar[dd] &&&&*{}\ar@{.>}[dd] &[132]\\
&&&&&& \\
[312]&*{}\ar[drr] &&&&*{}\ar[dll] &[231] \\
&&&*{}&&& \\
&&&[321]&&& \\
}}

\def\PtroisA{\xymatrix@R=4pt@C=4pt{
&&&*{}\ar@{-}[dll]\ar@{.}[dr]\ar@{-}[r] &*{}\ar@{-}[drr]\ar@{-}[dll] &&& \\
&*{}\ar@{-}[ddl]\ar@{-}[r]&*{}\ar@{-}[dd] & &*{}\ar@{.}[ddl]\ar@{.}[drr]&&*{}\ar@{-}[dd]\ar@{-}[dr]& \\
&&& & &&*{}\ar@{.}[r]\ar@{.}[ddl]&*{}\ar@{-}[dd] \\
*{}\ar@{-}[dd]\ar@{.}[dr]&&*{}\ar@{-}[ddl]\ar@{-}[drr]&*{}\ar@{.}[dll]\ar@{.}[drr] & &&*{}\ar@{-}[dll]\ar@{-}[dr]& \\
&*{}\ar@{.}[dd]&& &*{}\ar@{-}[ddl] &*{}\ar@{.}[dd]&&*{}\ar@{-}[ddl] \\
*{}\ar@{-}[dr]\ar@{-}[r]&*{}\ar@{-}[drr]&& & &&& \\
&*{}\ar@{-}[drr]&&*{}\ar@{-}[dr] & &*{}\ar@{.}[dll]\ar@{.}[r]&*{}\ar@{-}[dll]& \\
&&&*{}\ar@{-}[r] &*{} &&& \\
}}

\def\ShuffleTree{
\xymatrix@R=8pt@C=8pt{
0&1&5&9&2 &4 &6&7&8 \\
&&&&  &&&*{}\ar@{-}[lu]\ar@{-}[u]\ar@{-}[ur]&  \\
&&*{}\ar@{-}[lluu]\ar@{-}[luu]\ar@{-}[uu]\ar@{-}[uur]&&  &&*{}\ar@{-}[luu]\ar@{-}[ur]&&  \\
&&&&*{}\ar@{-}[llu]\ar@{-}[uuu] \ar@{-}[urr] \ar@{-}[d] &&&&  \\
&&&&  &&&&  \\
}}

\DeclareMathOperator*{\Sur}{Sur}
\DeclareMathOperator*{\Surj}{Surj}
\DeclareMathOperator*{\Sh}{Sh}

\journal{Journal of Combinatorial Theory A}

\begin{document}

\begin{frontmatter}



\title{Permutads}


\author{Jean-Louis Loday}
\ead{loday@math.unistra.fr}
\address{Institut de Recherche Math\'ematique Avanc\'ee,
    CNRS et Universit\'e de Strasbourg, France}
    
\author{Mar\'ia Ronco}
\ead{mariaronco@inst-mat.utalca.cl}
\address{Instituto de Matem\' aticas y F\' \i sica, Universidad de Talca, Chile}

\begin{abstract}
We unravel the algebraic structure which controls the various ways of computing the word $((xy)(zt))$ and its siblings. We show that it gives rise to a new type of operads, that we call permutads. A permutad is an algebra over the monad made of surjective maps between finite sets. It turns out that this notion is equivalent to the notion of ``shuffle algebra'' introduced previously by the second author. It is also very close to the notion of ``shuffle operad'' introduced by V.\ Dotsenko and A.\ Khoroshkin. It can be seen as a noncommutative version of the notion of nonsymmetric operads. We show that the role of the associahedron in the theory of  operads is played by the permutohedron in the theory of permutads.

\end{abstract}



\end{frontmatter}

\maketitle

\section*{Introduction} \label{s:int} The nonsymmetric operads are encoded by the planar rooted trees in the following sense. These trees determine a monad on the category of arity graded modules and a nonsymmetric operad is an algebra over this monad. In this paper we replace the planar rooted trees by the surjection maps between finite sets. We still get a monad on arity graded modules, and an algebra over this monad is called a \emph{permutad}.

A permutad can be presented with ``partial operations'' $\circ_{\ss}$ where $\ss$ is a shuffle, as generators. Under this presentation we see that the notion of permutad is equivalent to the notion of ``shuffle algebra'' introduced in \cite{Ronco11} by the second author. There is still another way of presenting a permutad, through some particular shuffle trees. This presentation enables us to to relate the notion of permutad to the notion of ``shuffle operad'' introduced by Dotsenko and Khoroshkin in \cite{DotsenkoKhoroshkin09}. 

If we restrict ourself to permutads generated by binary operations, then the general operations can be described by the ``leveled planar binary rooted trees''. These binary permutads are the relevant tool to handle algebras for which a product like  $((xy)(zt))$ depends on the first product which is performed: either $(xy)$ or $(zt)$.

An interesting  example of permutad is presented by one binary operation $xy$ subject to one relation
$$x(yz) = q \ (xy)z\ ,$$
where $q$ is a parameter in the ground field. In the nonsymmetric operad framework such an operad is interesting only for $q=0, 1$ or $\infty$, since in the other cases the free algebra collapses.  However, in the permutad setting, it makes sense for any value of $q$ and, computing in this framework, leads to the appearance of the length function of the Coxeter group $\Sy_{n}$. This result is a consequence of an independent interesting combinatorial result on the poset structure (weak Bruhat order) of the symmetric group $\Sy_{n}$. In this poset there are two kinds of covering relations which are best understood by replacing the permutations by the leveled planar binary trees. A covering relation of the first kind is obtained by moving a vertex from left to right. A covering relation of the second  kind is obtained by exchanging two levels. Our result says that the graph obtained by keeping only the covering relations of the first kind is connected.

As a consequence of the preceding result the associative permutad, obtained for $q=1$ and denoted by ${\permAs}$, is the same as the associative nonsymmetric operad $As$, i.e.\ ${\permAs}_{n}$ is one dimensional. However looking at the minimal model in nonsymmetric operads vs permutads leads to different objects. It is well-known that in the nonsymmetric operad setting the minimal model of $As$ is completely determined by the family of associahedrons. We prove that  in the permutad setting the minimal model of ${\permAs}$ is completely determined by the family of permutohedrons. 

The notion of permutad can be thought of as a certain ``noncommutative'' variation of the notion of nonsymmetric operad since we, essentially, leave out the parallel composition axiom for partial operations. This noncommutative feature is illustrated by the computation of the permutad encoding the associative algebras with derivation. In the nonsymmetric operad framework it involves the algebra of polynomials and in the permutad framework it involves the algebra of noncommutative polynomials.

\medskip
\noindent{\sc Contents}

1. Surjective maps and permutads

2. Partial operations and  nonsymmetric operads

3. Shuffle algebras

4. Binary quadratic permutads

5. Associative permutad and the permutohedron

6. The permutad of associative algebras with derivation

7. Permutads and shuffle operads

8. Appendix 1: Combinatorics of  surjections on finite sets

9. Appendix 2: The permutohedron

\bigskip

\noindent{\bf Notation.} Let $\u{n}=\{1,\ldots , n\}$ be a finite ordered set with $n$ elements. By convention $\u{0}$ is the empty set. The automorphisms group of $\u{n}$ is the symmetric group denoted by $\Sy_{n}$. Its unit element is denoted by $1_{n}$ (identity permutation). For the operadic  terminology the reader may refer either to \cite{MSS} or to \cite{LV}, from which .

\bigskip

\noindent{\bf Acknowledgement.} We warmly thank Vladimir Dotsenko for his comment on shuffle algebras. Thanks also to Emily Burgunder for her comments on the first draft of this paper. We are grateful to the referees for their comments and questions which help us to improve this paper.

{This work has been partially supported by  the Proyecto Fondecyt Regular 1100380, and the ``MathAmSud'' project OPECSHA 01-math-10.

\section{Surjective maps and permutads}\label{surj+perm}

We introduce ad hoc terminology about surjective maps of finite sets. Then we give the ``combinatorial definition'' of a permutad.

\subsection{Surjective maps}\label{surjmaps} Let $n$ and $k$ be positive integers such that $1\leq k\leq n$. We denote by $\Sur(\u{n}, \u{k})$ the set of surjective maps $t:\u{n}\to \u{k}$. We call  \emph{vertices} the elements in the target set $\u{k}$. The set of vertices of the surjective map $t$ is denoted by $\vert\, t$. The \emph{arity} of $v\in \vert\, t$ is 
$|v|:=\# t^{-1}(v)+1$ (note the shift). When $k=1$, $\Sur(\u{n}, \u{1})$ contains only one element that we call a \emph{corolla} (denoted by $c_{n+1}$), and when $k=n$ the set $\Sur(\u{n}, \u{n})$ coincides with the symmetric group $\Sy_{n}$. We adopt the following notation:
$\Surj_{n}:= \bigcup _{k} \Sur(\u{n}, \u{k})$ (disjoint union). By convention $\Surj_{0}= \{{\bf 1}\}$ and its unique element is considered as a surjective map ${\bf 1}:\u{0} \to \u{0}$ (formally $\u{0}=\emptyset$) with no vertex. In section \ref{appendixassoc} we list all the surjective maps for $n=2, 3$.

\subsection{Substitution}\label{substitution} Let $t\in \Sur(\u{n}, \u{k})$ and  $t_{i}\in \Sur(\u{i_{j}}, \u{m_{j}}), j=1,\ldots , k$, be surjective maps such that $i_{j}= \# t^{-1}(j)$. We put $m:=\sum_{j}m_{j}$. By definition the \emph{substitution} of $\{t_{j}\}$ in $t$ is the surjective map
$(t; t_{1},\ldots , t_{k})\in \Sur(\u{n}, \u{m})$ given by

$$(t; t_{1},\ldots , t_{k})(a) := m_{1}+\cdots+m_{j-1}+t_{j}(b),$$
whenever $t(a) = j$ and $a$ is the $b$th element in $t^{-1}(j)$.

 Example with $k=2$:
 
 $$\dessinA$$

Observe that substitution is an associative operation. 

The family of surjective maps endowed with the process of substitution could be described as a colored operad (cf.\ \cite{vanderLaan}), where the colors are integers, and also as a ``multi-category" (cf.\ \cite{Leinster}), see \ref{coloredop}.

\subsection{A monad on arity-graded modules}\label{monadPerm} Let $\KK$ be a commutative ring and let $\NN^+\Mod$ be the category of positively graded $\KK$-modules. An object $M$ of $\NN^+\-\Mod$ is a family $\{M_{n}\}, n\geq 1$, of $\KK$-modules. In this paper we say ``arity'' in place of degree for these objects.

We define a monad in the category $\NN^+\-\Mod$ of arity-graded modules as follows. First, for any $M$ and any surjective map $t\in \Sur(\u{n}, \u{k})$ we define a module
$$M_{t}:= \bigotimes _{v\in \vert{\, t}} M_{|v|},$$
where $|v|$ is the arity of the vertex $v$.

Second, for any arity-graded module $M$ we define an arity-graded module $\PPP(M)$ as follows:
$$\mathbb{P} (M)_{n+1}:= \bigoplus_{t\in \Surj_{n}} M_{t}\quad \mathrm{for} \quad n\geq 1,$$
and $\mathbb{P} (M)_{1}:= \KK\, \id$.

Explicitly the module $\PPP(M)_{n+1}$ is spanned by surjective maps with $n$ inputs whose vertices are decorated by elements of $M$.

\textsf{Example:} 
$$\xymatrix{
 & *{\bullet}\ar@{-}[dr] & *{\bullet}\ar@{-}[drr] & *{\bullet}\ar@{-}[dl] &*{\bullet}\ar@{-}[dll] & *{\bullet}\ar@{-}[dl] \\
 && \times\atop\mu\in M_{4}& &\times\atop\nu\in M_{3} & \\
 }$$
So, we have defined a functor
$$\PPP: \NN^+\Mod\to \NN^+\Mod.$$
There is a natural map $\iota(M) : M\to \PPP(M)$ which consists in associating to an element $\mu\in M_{n}$ the corolla $c_{n}$ of arity $n$ decorated by $\mu$.

\begin{prop}\label{propMonad} The substitution of surjective maps defines a transformation of functors $\Gamma : \mathbb{P}\circ  \mathbb{P} \to \mathbb{P} $ which is associative and unital. So $(\mathbb{P} ,\Gamma , \iota)$ is a monad on arity-graded modules.
\end{prop}
\begin{proo} From the definition of $ \mathbb{P} $ we get
\begin{alignat*}{2}
 \mathbb{P} ( \mathbb{P} (M))_{n} &= \bigoplus_{t\in \Surj_{n-1}}  \mathbb{P} (M)_{t}\\
&=  \bigoplus_{t\in \Surj_{n-1}} \big( \mathbb{P} (M)_{i_{1}}\t \cdots \t \mathbb{P} (M)_{i_{k}}\big)\\
  &=  \bigoplus_{t\in \Surj_{n-1}} \Big( \bigotimes _{j=1}^{j=k}\big(\bigoplus_{s\in \Surj_{j-1}} M_{s}\big)\Big).
\end{alignat*}

  Under the substitution of surjective maps we get an element of $ \mathbb{P} (M)_{n}$, since at any vertex $j$ of $t$ we have an element of $ \mathbb{P} (M)_{i_{j}}=\bigoplus_{s\in \Surj_{j-1}} M_{s}$, that is a surjective map $s$ and its decoration. We substitute this data at each vertex of $t$ to get a new decorated surjective map, decorated by elements of $M$.  Therefore we have obtained a linear map $\Gamma(M): \mathbb{P}(\mathbb{P}(M))_{n} \to \mathbb{P}(M)_{n}$, which defines a morphism  of arity-graded modules
  $$\Gamma(M): \mathbb{P}(\mathbb{P}(M)) \to \mathbb{P}(M).$$
  Since it is functorial in $M$ we get a transformation of functors $\Gamma: \mathbb{P}\circ \mathbb{P}\to \mathbb{P}$.

  Since the substitution process is associative, $\Gamma$ is associative. Substituting a vertex by a corolla does not change the surjective map. Substituting a surjective map to the vertex of a corolla gives the former surjective map. Hence $\Gamma$ is also unital.

  We have proved that  $(\mathbb{P} ,\Gamma , \iota)$ is a monad.
 \end{proo}

\subsection{Colored operad}\label{coloredop} A \emph{colored operad} is to an operad what a category is to a monoid. More precisely, there is a set of colors and for each operation of the operad there is a color assigned to each input and a color assigned to the output. In order for a composition like $\mu \circ (\nu_{1}, \ldots , \nu_{k})$ to hold the color of the ouput of $\nu_{i}$ has to be equal to the color of the $i$th input of $\mu$. Of course the colors of the inputs of the composite are the colors of the $\nu_{i}$'s and the color of the output is the color of the output of $\mu$. See for instance \cite{vanderLaan}.

The monad $(\PPP, \Gamma, \iota)$ defined above can be seen as a nonsymmetric colored operad, where the colors are the natural numbers. 

\subsection{Permutads}\label{defPermutad} By definition a \emph{permutad} $\PP$ is a unital algebra over the monad $(\mathbb{P} ,\Gamma , \iota)$.  So $\PP$ is an arity-graded module such that $\PP_{1}=\KK\, \id$  endowed with a morphism of arity-graded modules 
$$\Gamma_{\PP}:\mathbb{P}(\PP)\to \PP$$
compatible with the composition $\Gamma$ and the unit $\iota$.
It means that the following diagrams are commutative:
$$\xymatrix{
                                                                                        &\mathbb{P}  (\mathbb{P}  (\PP ))  \ar[rr]^-{\PPP(\Gamma_{\PP})}&&\mathbb{P} (\PP) \ar[ddd]^{\Gamma_{\PP}}\\
(\mathbb{P}  \circ \mathbb{P}  )(\PP) \ar[ur]^-{\cong} \ar[dd]_{\Gamma(\PP)}&                                                                  &&\\
&&\\
\mathbb{P} (\PP)  \ar[rrr]^{\Gamma_{\PP}} &&& \PP  \\
}$$
\noindent and
$$\xymatrix{
\Id(\PP) \ar[r]^-{\iota(\PP)} \ar[dr]_{=}&\mathbb{P} (\PP)\ar[d]^-{\Gamma_{\PP}} \\
& \PP  }$$
We often call an element of $\PP$ an \emph{operation} and the map $\Gamma(\PP)$ the composition of operations. We, now, give a first example: the free permutad.

\begin{prop}[Free permutad]\label{propFreePermutad} For any arity-graded module $M$ such that $M_{1}=0$, the arity-graded module $\PPP(M)$ is a permutad which is the free permutad over $M$.  
\end{prop}
\begin{proo} The structure of permutad $\Gamma_{\PPP(M)}$ on $\PPP(M)$ is induced by $\Gamma$, that is 
$$\Gamma_{\PPP(M)}: \mathbb{P}(\mathbb{P}(M) = (\mathbb{P}\circ \mathbb{P})(M) \xrightarrow{\Gamma(M)} \mathbb{P}(M).$$

The proof that $\mathbb{P}(M)$ is free among permutads goes as follows. Let $\QQQ$ be another permutad and let $f: M\to \QQQ$ be a morphism of arity-graded modules. The composite $\mathbb{P}(M)\xrightarrow{\PP(f)} \mathbb{P}(\QQQ) \to \QQQ$, which uses the permutadic structure of $\QQQ$, extends the map $f$. It is straightforward to check that it is a map of permutads and that it is unique as a permutad morphism extending $f$. So $\PPP(M)$ is free over $M$.
\end{proo}

\subsection{Ideal and quotient} Given a permutad $\PP$, a sub-module $\mathcal{I}$ is said to be an ideal if any permutadic composition of operations in $\PP$ for which at least one is in  $\mathcal{I}$, is also in  $\mathcal{I}$. As an immediate consequence the quotient $\PP/ \mathcal{I}$ acquires a structure of permutad.

\subsection{Differential graded permutad}\label{dgpermutad} By replacing the category of modules over the ground ring $\KK$ by the category of differential graded modules (i.e.\ chain complexes), we obtain the definition of \emph{differential graded permutads}, abbreviated into \emph{dg permutads}. Explicitly, when $(M,d)$ is a dg 
$\NN^+\Mod$ we make $\PPP(M)$ into a dg $\NN^+\Mod$ by defining the differential $d$ by
$$d(t; \mu_{1}, \ldots , \mu_{k}) := \sum_{i=1}^{k} (-1)^{\epsilon_{i}}(t; \mu_{1}, \ldots , d\mu_{i}, \ldots , \mu_{k})$$
where $t: \u{n} \to \u{k}$ is a surjective map and $\epsilon_{i}= |\mu_{1}| + \cdots +  |\mu_{i-1}|$. Then the structure map $\Gamma_{M}: \PPP(M) \to M$ is required to be a dg $\NN^+\Mod$ morphism.


\section{Partial operations and non-symmetric operads}\label{partialshufflealg} The definition of a permutad that we have given is similar to the so-called ``combinatorial'' presentation of an operad (see Chapter 5 of \cite{LV}). Any surjection between finite sets can be obtained by successive substitutions of surjections with target set of size $2$. This property will enable us to give a definition of a permutad with a minimal data. It is similar to the so-called ``partial'' presentation of an operad.

\subsection{On the substitution of surjective maps with target set of size $2$} Let us consider a surjective map $r:\u{k}\to \u{3}$. We denote by $n+1, m+1, \ell+1$ the arity of the vertex $1,2,3$ respectively. There are two ways to consider $r$ as a result of substitution. Either as a substitution which gives rise to  $1$ and $2$, or, as a substitution which gives $2$ and $3$.
The first process gives two surjective maps $s:\u{m+\ell} \to \u{2}$ and $t: \u{n+m+\ell} \to \u{2}$, and the second process gives 
$v:\u{n+m} \to \u{2}$ and $u: \u{n+m+\ell} \to \u{2}$.

\medskip

$\xymatrix@R=8pt@C=8pt{
\bullet \ar@{-}[rddd]&  \bullet  \ar@{-}[rrddd] &\bullet \ar@{-}[lddd]&\bullet \ar@{-}[rrddd]&\bullet \ar@{-}[lllddd]&\bullet \ar@{-}[ddd]&\bullet \ar@{-}[lllddd]\\
&&&&&&\\
&&&&&&\\
  & 1 & & 2 &&3 & \\
 }$
 \hskip2cm $\xymatrix@R=8pt@C=8pt{
\bullet \ar@{-}[rddd]&  \bullet  \ar@{-}[rrddd] &\bullet \ar@{-}[lddd]&\bullet \ar@{-}[rrddd]&\bullet \ar@{-}[lllddd]&\bullet \ar@{-}[ddd]&\bullet \ar@{-}[lllddd]\\
&&&&&&\\
&&&&&&\\
  & 1 & & 2 &&3 & \\
 }$\\
 ${}\hskip1cm\underbrace{\hskip2cm}_{v}$\hskip4cm  ${}\hskip1.6cm\underbrace{\hskip2cm}_{s}$\\
 ${}\hskip1cm\underbrace{\hskip3cm}_{u}$ \hskip2.4cm  ${}\hskip1cm\underbrace{\hskip3cm}_{t}$
 
 \medskip
 
Hence $r$ can be seen either as the substitution of $s$ in $t$ at $2$, or as a substitution of $v$ in $u$ at $1$. So, if $\lambda, \mu, \nu$ are the decorations of the vertices $3$, $2$ and $1$ respectively, we get
$$(\lambda\circ_{s}\mu) \circ_{t}\nu =  \lambda\circ_{u}(\mu \circ_{v}\nu).$$

\begin{lemma}\label{Lemmadiamond}  For any operations $\lambda\in \PP_{\ell+1}, \mu \in \PP_{m+1}, \nu \in \PP_{n+1}$  one has 
$$(\lambda\circ_{s}\mu) \circ_{t}\nu =  \lambda\circ_{u}(\mu \circ_{v}\nu).$$
\end{lemma}

\begin{proo} Both terms of the expected equality are equal to the surjective map obtained by decorating the vertex $3$ by $\lambda $, $2$ by $\mu $ and  $1$ by $\nu $ in $r$. This is an immediate consequence the associativity property of substitution for surjective maps.
\end{proo}

\subsection{Partial presentation of a permutad}\label{partialOp} Let $(\PP, \Gamma, \iota)$ be a permutad (so we suppose $\PP_{1}=\KK\, \id$). Any surjection $t$ with target set of size $2$ (i.e.\ $t:\u{m+n} \to \u{2}$) determines a linear map that we denote by
$$\circ_{t}: \PP_{m+1}\t \PP_{n+1}\to \PP_{m+n+1},$$
where $n= \#t^{-1}(1)$.

\begin{thm}\label{thmpartial} A permutad $(\PP, \Gamma, \iota)$ is completely determined by the arity graded module $\{\PP_{n}\}_{n\geq 1}$, with $\PP_{1}=\KK$, and the partial operations 
$$\circ_{t}: \PP_{m+1}\t \PP_{n+1}\to \PP_{m+n+1},\quad t:\u{m+n}\to \u{2}, \ n= \#t^{-1}(1),$$
satisfying
$$(\diamondsuit)\qquad (\lambda\circ_{s}\mu) \circ_{t}\nu =  \lambda\circ_{u}(\mu \circ_{v}\nu),$$
for any surjective map $r$ with target $\u{3}$  such that $ (t;c_{n+1},s)=r= (u;v,c_{l+1})$, for $\#r^{-1}(1)=n$ and $\# r^{-1}(3)=l$.
\end{thm}

\begin{proo} Let us start with a permutad $(\PP, \Gamma, \iota)$. We define the partial operations  by $\mu\circ_{t}\nu := \Gamma(t ; \nu, \mu)$. 
Formula $(\diamondsuit)$ has been proved in Lemma \ref{Lemmadiamond}.

On the other hand, let us start with the arity graded module $\PP$ and the partial operations $\circ_t$, for any surjective map $t$, satisfying $(\diamondsuit )$. 

For any surjective map $t:{\underline n}\longrightarrow {\underline r}$, with $r\geq 2$, we introduce new maps $v_t$ and ${\tilde {t}}$ as follows:\begin{enumerate}
\item  $v_t:{\underline {n_1}}\longrightarrow {\underline {r-1}}$ is  the surjective map given by $v_t(i):=t(s_i),$
where $n_1:=n-\#t^{-1}(r)$ and ${\underline n}\setminus t^{-1}(r)=\{ s_1<\dots <s_{n_1}\}$.
\item ${\tilde t}:{\underline n}\longrightarrow {\underline 2}$ is given by the formula: $${\tilde t}(i)=\begin{cases}1,&\ {\it for}\ t(i)<r,\\
2,&\ {\it for}\ t(i)=r.\end{cases}$$\end{enumerate}

For $r>2$, we get that $t=({\tilde t}; v_t, c_{|r|})$, with $|r| =\# t^{-1}(r)+1.$
\medskip

We define $\Gamma (t;\lambda _1,\dots ,\lambda _r)$ for any surjection $t:{\underline n}\longrightarrow {\underline r}$ and any family of elements $\lambda _i\in \PP _{\# t^{-1}(i) +1}$, $1\leq i\leq r$ as follows:

If  $r=1$, then $t=c_{n+1}$. So, $\Gamma (t;\lambda _1):=\lambda _1$.

For $r=2$, we apply the partial operation $\circ _t$ to get $\Gamma (t;\lambda _1,\lambda _2):=\lambda _2\circ _{t}\lambda _1$.

For $r>2$ and $t=({\tilde t}; v_t, c_{|r|})$,  we define $$\Gamma (t;\lambda _1,\dots ,\lambda _r):=\lambda _r\circ _{\tilde {t}}\Gamma (v_t;\lambda _1,\dots ,\lambda _{r-1}).$$

In order to check that $\PP $ is a permutad, we have to verify that $\Gamma$ satisfies the equality:
$$\Gamma ((t;t_1,\dots ,t_r);\lambda _1^1,\dots \lambda _{k_1}^1,\dots ,\lambda _{k_r}^r)=\Gamma (t;\Gamma (t_1;\lambda _1^1,\dots ,\lambda _{k_1}^1),\dots,
\Gamma (t_r;\lambda _1^r,\dots ,\lambda _{k_r}^r)),$$
for any family of surjective maps $t:{\underline n}\longrightarrow {\underline r}$ and $t_i:{\underline {\#t^{-1}(i)}}\longrightarrow {\underline {k_i}}$, for $1\leq i\leq r$, and any collection of elements $\{\lambda _i^j\in \PP_{k_j+1}\mid 1\leq i\leq k_j\ {\rm and}\ 1\leq j\leq r\}$. 
\bigskip

To prove it we  use a double recursive argument on $r$ and on $k_r$.

If $r=1$, the result is evident because $\Gamma ((t;t_1),\lambda )=\Gamma (t_1;\lambda )=\Gamma (t: \Gamma (t_1;\lambda))$.
\medskip

For $r\geq 2$, we have:
\begin{enumerate}
\item  $t=({\tilde t};v_t,c_{|r|})$, 
\item $t_r=({\tilde {t_r}}; v_{t_r},c_{|k_r|})$.   \end{enumerate} 

Let us denote by $w$ the element $(t;t_1,\dots ,t_r)$. The map ${\tilde {w}}:{\underline n}\longrightarrow {\underline 2}$ is given by:
$${\tilde w}(i)=\begin{cases}1,&\ {\rm if}\ t(i)<r,\ {\rm or\ if}\ i=s_j^r\in t^{-1}(r)\ {\rm and}\ t_r(j)<k_r,\\
2,&\ {\rm if}\ i=s_j^r\in t^{-1}(r)\ {\rm and}\ t_r(j)=k_r,\end{cases}$$
where $t^{-1}(r)=\{s_1^r<\dots <s_{\#t^{-1}(r)}^r\}$. So, we get $w=({\tilde {w}}, v_w,c_{|k_r|})$, with:
$$v_w(i)=t_l(j),\ {\rm if}\ i=s_j^l\in t^{-1}(l),$$
where $t^{-1}(l)=\{ s_1^l<\dots <s_{\#t^{-1}(l)}^l\}$, for $1\leq l\leq r$.

If $k_r=1$,  then $t_r=c_{|r|}$, ${\tilde w}={\tilde t}$ and $v_w=v_t$.  So, applying the recursive hypothesis on $r$, we have:
$$\displaylines{
\Gamma((t;t_1,\dots ,t_r),\lambda _1^1,\dots ,\lambda _1^r)=\lambda _1^r\circ _{\tilde t} \Gamma ((v_t;t_1,\dots ,t_{r-1});\lambda _1^1,\dots ,\lambda _{k_{r-1}}^{r-1})=\hfill\cr
\lambda _1^r\circ _{\tilde t} \Gamma (v_t;\Gamma (t_1;\lambda _1^1,\dots ,\lambda _{k_1}^1),\dots ,\Gamma (t_{r-1};\lambda _1^{r-1},\dots ,\lambda _{k_{r-1}}^{r-1}))=\cr
\hfill \Gamma (t;\Gamma (t_1;\lambda _1^1,\dots ,\lambda _{k_1}^1),\dots ,\Gamma (t_{r-1};\lambda _1^{r-1},\dots ,\lambda _{k_{r-1}}^{r-1}),\Gamma (t_r;\lambda _1^r)).\cr}$$
\medskip

For $k_r>1$, we have:
$$\displaylines {
\Gamma((t;t_1,\dots ,t_r),\lambda _1^1,\dots ,\lambda _{k_r}^r)=\Gamma (({\tilde w};v_w,c_{|k_r|});\lambda _1^1,\dots ,\lambda _{k_r}^r)=\hfill\cr
\lambda _{k_r}^r\circ _{\tilde w} \Gamma (v_w;\lambda _1^1,\dots ,\lambda _{k_r-1}^r).\cr }$$

The set $t^{-1}(\{1,\dots ,r-1\})\subset w^{-1}(\{1,\dots ,k_1+\dots +k_r-1\})=\{s_1<\dots <s_p\}$, and it is not difficult to see that $$v_w=(u; t_1,\dots ,t_{r-1},{\hat {t_r}}),$$
where $u(i)=t(s_i)$, for $1\leq i\leq p$. Moreover, if $t_r^{-1}(\{ 1,\dots ,k_r-1\})=\{l_1<\dots <l_q\}$, then ${\hat {t_r}}(j)=t_r(l_j)$, for $1\leq j\leq q$.
\medskip

Let $u=({\tilde {u}}; v_u, c_{q+1})$. A straightforward computation shows that:\begin{enumerate}
\item $({\tilde {w}};{\tilde {u}},c_{\#t_r^{-1}(k_r)})=({\tilde {t}}; c_{n-\#t^{-1}(r)},{\tilde {t_r}}),$
\item $v_u=v_t$.\end{enumerate}
\medskip

Now, the recursive argument on $r$ and $k_r$, implies that:
$$\displaylines {
\Gamma (v_w;\lambda _1^1,\dots ,\lambda _{k_r-1}^r)=\hfill\cr
\Gamma (u;\Gamma (t_1;\lambda _1^1,\dots ,\lambda _{k_1}^1),\dots ,\Gamma (t_{r-1};\lambda _1^{r-1},\dots ,\lambda _{k_{r-1}}^{r-1}),\Gamma ({\hat {t_r}};\lambda _1^r,\dots ,\lambda _{k_r-1}^r))=\cr
\Gamma ({\hat {t_r}};\lambda _1^r,\dots ,\lambda _{k_r-1}^r)\circ {\tilde{u}}\Gamma (v_t;\Gamma (t_1;\lambda _1^1,\dots ,\lambda _{k_1}^1),\dots ,\Gamma (t_{r-1};\lambda _1^{r-1},\dots ,\lambda _{k_{r-1}}^{r-1})).\cr }$$
\medskip

Since $({\tilde {w}};{\tilde {u}},c_{\#t_r^{-1}(k_r)})=({\tilde {t}}; c_{n-\#t^{-1}(r)},{\tilde {t_r}})$, condition $(\diamondsuit )$, states that:
$$\displaylines {
\lambda _{k_r}^r\circ _{\tilde w} (\Gamma ({\hat {t_r}};\lambda _1^r,\dots ,\lambda _{k_r-1}^r)\circ _{\tilde{u}}\Gamma (v_t;\Gamma (t_1;\lambda _1^1,\dots ,\lambda _{k_1}^1),\dots ,\Gamma (t_{r-1};\lambda _1^{r-1},\dots ,\lambda _{k_{r-1}}^{r-1})))=\hfill\cr
(\lambda _{k_r}^r\circ _{\tilde {t_r}} \Gamma ({\hat {t_r}};\lambda _1^r,\dots ,\lambda _{k_r-1}^r))\circ _{\tilde{t}}\Gamma (v_t;\Gamma (t_1;\lambda _1^1,\dots ,\lambda _{k_1}^1),\dots ,\Gamma (t_{r-1};\lambda _1^{r-1},\dots ,\lambda _{k_{r-1}}^{r-1}))=\cr
\Gamma (t_r;\lambda _1^r,\dots ,\lambda _{k_r}^r)\circ _{\tilde{t}}\Gamma (v_t;\Gamma (t_1;\lambda _1^1,\dots ,\lambda _{k_1}^1),\dots ,\Gamma (t_{r-1};\lambda _1^{r-1},\dots ,\lambda _{k_{r-1}}^{r-1}))=\cr
\hfill\Gamma (t; \Gamma (t_1;\lambda _1^1,\dots ,\lambda _{k_1}^1),\dots ,\Gamma (t_r;\lambda _1^r,\dots ,\lambda _{k_r}^r),\cr }$$
which ends the proof.
\end{proo}

\subsection{The partial operations $\circ_i$} Let now $t:\u{m+n}\to \u{2}$ be such that the inverse image of the vertex $1$ is made of $n$ consecutive elements $t^{-1}(1)=\{i, i+1, \ldots , i+n-1\}$. So, once $m$ and $n$ have been chosen, $t$ is completely determined by the integer $i$. We will sometimes denote it by $\circ_{i}$ instead of $\circ_{t}$ because it has properties  similar to the partial operations in the operad framework (see \cite{LV} or \cite{MSS}). More precisely the permutadic operation $\circ_{i}:\PP_{m+1}\t \PP_{n+1} \to \PP_{m+n+1}$ is similar to the operadic operation which corresponds to the tree obtained by grafting a corolla on the $i$th leaf of another corolla:

${\xymatrix@R=8pt@C=4pt{
1&&i& &i+n-1&&m+n\\
\bullet\ar@{-}[rrrrddd]&\cdots&\bullet\ar@{-}[ddd]&\cdots &\bullet\ar@{-}[llddd]&\cdots &\bullet\ar@{-}[llddd]\\
&&&&&&\\
&&&&&&\\
&&\ \times \nu&&\ \times \mu&&\\
&&1&&2&&\\
&&&\textrm{surjection}&&&
}}$\qquad 
${\xymatrix@R=8pt@C=8pt{
&&&&&&&\\
*{}\ar@{-}[rrrddd]&*{}\ar@{-}[rrddd]&*{}\ar@{-}[rd]&*{}\ar@{-}[d]&*{}\ar@{-}[ld]  &&*{}\ar@{-}[lllddd]\\
&&&\nu\ar@{-}[dd]&&&\\
&&&\quad i&&&\\
&&&\mu\ar@{-}[d]&&&\\
&&&*{}&&&\\
&&&\textrm{tree}&&
}}$

When $n=2$ the surjective maps (that is the permutations) $1_{2}$ and $(12)$ correspond respectively to the operations $\circ_{1}$ and $\circ_{2}$. For $n=3$ there is only one surjective map $\u{3} \to \u{2}$ which is not of the type $\circ_{i}$, it is
$$\xymatrix@R=8pt@C=8pt{
\bullet \ar@{-}[rdd]& & \bullet  \ar@{-}[rdd]& &\bullet \ar@{-}[llldd]\\
&&&&\\
 & \times & & \times & \\
 }$$
Under the bijection between surjective maps and the cells of the permutohedron, it corresponds to the dotted arrow in the hexagon (see Figure \ref{figuregraphtrees}).

\begin{prop}\label{sequentialaxiom} The  partial operations $\circ_i$  in a permutad $\PP$ satisfy the \emph{sequential composition relation}:

 \begin{displaymath}
\begin{array}{crcll}
&(\lambda \circ_i \mu)\circ_{i-1+j}\nu  &=& \lambda \circ_i (\mu\circ_{j}\nu ), &\mathrm{for }\ 1\leq i\leq l, 1\leq j\leq m,  \\
\end{array}
\end{displaymath}
for any $\lambda \in \PP_l, \mu\in \PP_m, \nu\in \PP_n$.
\end{prop}

\begin{proo} This is a particular case of the formula $(\diamondsuit)$ in Theorem \ref{thmpartial}.
\end{proo}
Observe that, in a permutad, the partial operations $\circ_i$ do not necessarily satisfy the parallel composition relation (see \cite{LV}, Chapter 5).

\subsection{Nonsymmetric operad} Let us recall that a nonsymmetric operad (we often write ns operad for short) can be defined as we defined a permutad but with planar trees in place of surjections. Here we consider only the trees which have at least two inputs at each vertex. The monad of planar trees is denoted by $P\TT$ and a ns operad is an algebra over this monad (cf.\ for instance the combinatorial definition of a ns operad in \cite{LV} section 5.8.5). 

A ns operad can also be described by means of the partial operations $\circ_i$ which are required to satisfy, not only the sequential composition axiom $(\diamondsuit)$ but also the parallel composition axiom which reads:
\begin{displaymath}
\begin{array}{rcll}
(\lambda \circ_i \mu)\circ_{k-1+m}\nu  &=& (\lambda \circ_k\nu)\circ_{i}\mu , &\mathrm{for }\  1\leq i < k\leq l, \\
\end{array}
\end{displaymath}
for any $\lambda \in \PP(l), \mu\in \PP(m), \nu\in \PP(n)$. 

\subsection{Pre-permutad} We define a \emph{pre-permutad} as an arity-graded module $\PP$ with $\PP_1=\KK \id$ equipped with partial operations $\circ_i$ satisfying the sequential composition axiom $(\diamondsuit)$. From the previous discussion it follows that there are two forgetful functors
$$ \{\textrm{ns operads}\} \to  \{\textrm{pre-permutads}\} \leftarrow  \{\textrm{permutads}\} .$$
The notion of pre-permutad appeared first in \cite{Ronco11} under the name \emph{pre-shuffle algebra}.

We will see in section  \ref{BinaryPerm} that in the binary case we can construct a more direct relationship between ns operads and permutads.

\section{Shuffle algebras}

We make explicit the bijection between surjections and shuffles. Under this bijection we show that the notion of permutad (resp.\ pre-permutad) is equivalent to the notion of shuffle algebra (resp.\ pre-shuffle algebra) introduced previously by the second author in \cite{Ronco11}.

\subsection{Shuffles}\label{shuflle} By definition an \emph{$(i_{1}, \ldots , i_{k})$-shuffle} in $\Sy_{n}$, $n=i_{1}+\cdots + i_{k}$,  is a permutation $\ss$ such that for any $j=1, \ldots, k$ one has
$$\ss(i_{1}+\cdots +i_{j-1}+1)<\ss( i_{1}+ \cdots +i_{j-1}+2)<  \ldots < \ss( i_{1}+ \cdots +i_{j-1}+i_{j}).$$
For instance the $(1,2)$-shuffles are $[1|2,3], [2|1,3],[3|1,2]$. Following Stasheff let us call \emph{unshuffle} the inverse of a shuffle. So the  $(1,2)$-unshuffles are $[1,2,3], [2,1,3],[2,3,1]$.

\begin{lemma}\label{lemmasurjshuffle} There is a bijection between the set of shuffles $\Sh(i_{1}, \ldots , i_{k})\subset \Sy_{n}$ and the subset of $\Sur(\u{n}, \u{k})$ made of surjective maps $t:\u{n}\to \u{k}$ such that $i_{j}= \# t^{-1}(j)$.
\end{lemma}
\begin{proo} Starting with  a surjective map $t$ we construct a sequence of integers $\ss(1), \ldots , \ss(n)$ as follows: let $t^{-1}(j)=\{l_1^j<\dots <l_{i_j}^j\}$ be the inverse image of $j$ by $t$, for $1\leq j\leq k$. The sequence 
defines a permutation $\ss _t^{-1}$ whose image is:
$$(\ss _t^{-1}(1),\dots ,\ss _t^{-1}(n)):=(l_1^1,\dots ,l_{i_1}^1,l_1^2,\dots ,l_1^{k},\dots ,l_{i_k}^k),$$
 which is a $(i_{1}, \ldots , i_{k})$-shuffle by construction. It is immediate to check that we have a bijection as expected.\end{proo}

\textsf{ Example:}

surjective map $\u{5}\to \u{3}$ \hskip2cm $(3,2)$-shuffle \hskip2cm $(3,2)$-unshuffle

\bigskip

$\xymatrix{
  *{\bullet}\ar@{-}[d] & *{\bullet}\ar@{-}[drr] & *{\bullet}\ar@{-}[dll] & *{\bullet}\ar@{-}[dlll] & *{\bullet}\ar@{-}[dl]\\
  *{\times}& &&*{\times}& &      \\
  }$
 $[1,3,4 | 2,5]$\hskip 2.5cm $[1, 4, 2, 3, 5]$\ .
 
 \bigskip
 
Given a pair of permutations $(\ss ,\tau)\in \Sy_n\times \Sy_m$ there is a natural way to construct the concatenation $\ss\times \tau$ of them, which is an element of $\Sy_{n+m}$. This product extends naturally to a product $\times : \Sur(\u{n},\u{k})\times \Sur(\u{m},\u{h})\rightarrow \Sur(\u{n+m},\u{k+h})$, by setting:
$$t\times w(j):=\begin{cases} t(j),&\ {\rm for}\ 1\leq j\leq n\\
w (j-n)+k,&\ {\rm for}\ n+1\leq j\leq n+m.\end{cases}$$ 
The product $\times$ is associative.

A well-known result about shuffles (see for instance \cite{BBHT}), states that:
$$\Sh(i_1+i_2,i_3)\cdot (\Sh(i_1,i_2)\times 1_{\Sy_{i_3}})=\Sh(i_1,i_2,i_3)=\Sh(i_1,i_2+i_3)\cdot (1_{\Sy_{i_1}}\times \Sh(i_2,i_3)),$$
where $1_{\Sy_n}$ denotes the identity of the group $\Sy_n$ and $\cdot $ denotes the usual product in $\Sy_{i_1+i_2+i_3}$.

The paragraph above shows that any $(i_1,\dots ,i_k)$-shuffle $\ss$ may be written, in a unique way, as
$$\ss = \ss_1\cdot (\ss_2\times 1_{\Sy_{i_k}})\cdot \dots \cdot (\ss _{k-1}\times 1_{\Sy_{i_3+\dots +i_k}}),$$
with $\ss _j\in  \Sh(i_1+\dots +i_{k-j}, i_{k-j+1})$.

Note that surjections with target size ${\underline 2}$ correspond to shuffles of type $(i_1,i_2)$.
 
 \begin{prop}\label{shufflealgeb} Let $t:{\underline n}\rightarrow {\underline 2}$ be a surjective map and let $t_j\in \Sur(\u{i}_j,\u{m}_j)$, for $j=1,2$, be such that $\u{i}_j=\# t^{-1}(j)$, we have that,
 $$\ss _{(t;t_1,t_2)}=\ss _t\cdot (\ss _{t_1}\times \ss_{t_2}),$$
 where $\cdot $ denotes the composition of maps.\end{prop}
 
 \begin{proo} First, it is easy to check that $\ss _{t_1\times t_2}=\ss _{t_1}\times \ss _{t_2}$.
 
 Let $t^{-1}(1)=\{l_1,\dots ,l_{i_1}\}$ and $t^{-1}(2)=\{ h_1,\dots ,h_{i_2}\}$.  We have that:
 $$(t;t_1,t_2)(i)=\begin{cases} t_1(j),&{\rm for}\ i=l_j,\\
t_2(j)+m_1,&{\rm for}\ i=h_j.\end{cases}$$
Suppose that $t_1^{-1}(k)=\{ s_1^k,\dots ,s_{q_k}^k\}$, for $1\leq k\leq m_1$, and that $t_2^{-1}(k-m_1)=\{ r_1^k,\dots ,r_{p_k}^k\}$, for any $m_1+1\leq k\leq m_1+m_2$.\begin{itemize}
\item If $1\leq i\leq i_1$ is such that $q_1+\dots +q_{k-1}<i\leq q_1+\dots +q_k$, then $\sigma _{(t;t_1,t_2)}(i)=l_{s_{i-q_1-\dots -q_{k-1}}^k}$,
 \item if $i_1+1\leq i\leq n$ is such that $p_1+\dots +p_{k-1}<i-i_1\leq p_1+\dots +p_k$, then $\sigma _{(t;t_1,t_2)}(i)=h_{r_{i-i_1-p_1-\dots -p_{k-1}}^k}$.\end{itemize}
\medskip

On the other hand, note that for $1\leq i\leq i_1$, 

$\sigma _{t_1\times t_2}(i)=s_{i-q_1-\dots -q_{k-1}}^k$, when $q_1+\dots +q_{k-1}<i\leq q_1+\dots +q_k$,

while for $i_1<i\leq n$, 

$\sigma _{t_1\times t_2}(i)=r_{i-i_1-p_1-\dots -p_{k-1}}^k$, when $p_1+\dots +p_{k-1}<i-i_1\leq p_1+\dots +p_k$.
\medskip

Since $$\sigma _t(i)=\begin{cases}l_i,&\ {\rm for}\ 1\leq i\leq i_1,\\
h_{i-i_1},&\ {\rm for}\ i_1<i\leq n,\end{cases}$$ composing with $\sigma_t$ we get the expected result.
\end{proo}

\subsection{Shuffle algebra \cite{Ronco11}}\label{ShuffleAlg}  A \emph{shuffle algebra} is a graded $\KK$-module $A=\bigoplus_{n\geq 0}A_{n}$ such that $A_{0}=\KK\, 1$ endowed with binary operations
$$\bullet _{\gamma }:A_{n}\otimes A_{m}\rightarrow A_{n+m},\ {\rm for}\ \gamma \in \Sh(n,m),$$
verifying:
 $$(\ddag)\qquad x\bullet _{\gamma }(y\bullet _{\delta }z)=(x\bullet _{\sigma}y)\bullet _{\lambda }z,$$
whenever $(1_{n}\times \delta )\cdot \gamma= (\sigma \times 1_{r})\cdot \lambda$ in $\Sh(n,m,r)$. It is also supposed that $1$ is a unit on both sides. Since any $k$-shuffle $\sigma \in \Sh(i_{1}, \ldots , i_{k})$ can be written as a composition of $2$-shuffles, the above relation implies that for any such $\sigma $ there is a well-defined map 
$$\bullet_{\sigma }: A_{i_{1}}\t \cdots \t A_{i_{k}}\to A_{i_{1}+\cdots +i_{k}}.$$

The relationship with permutads is given by the following result.

\begin{prop}\label{equivPermShuffleAlg}  There is an equivalence between permutads $\PP$  and shuffle algebras $A$. It is given by $ \PP_{n+1}= A_{n}$,  $\circ_{t}= \bullet_{\sigma_t}$, where $t$ is a surjective map with target $\u 2$.
\end{prop}

\begin{proo} Lemma \ref{lemmasurjshuffle} gives a bijection between surjective maps $\u{n} \to \u{k}$ and  $k$-shuffles, which restricts to a bijection between surjective maps with target $\u 2$ and $2$-shuffles. 

Note that the identity $1_n\in \Sy_n$ corresponds via this bijection to the corolla $c_{n+1}\in \Sur(\u n,\u 1)$, that is $\ss _{c_{n+1}}=1_n$. So, Proposition \ref{shufflealgeb} implies that for any surjective map $r$ of target $\u 3$, such that $r=(t;s,c_{k+1})=(u;c_{j+1},v)$, we have that $\ss _r=(\ss _s\times 1_k)\cdot \ss _t=(1_j\times \ss _v)\cdot \ss _u$. 

From this relation, the relation $(\ddag)$ between shuffles used in the definition of a shuffle algebra corresponds, via the bijection, to the relation $(\diamondsuit)$ beween surjective maps proved in Theorem \ref{thmpartial}.
\end{proo} 

Several examples of shuffle algebras, and therefore of permutads have been given in \cite{Ronco11}.

\section{Binary quadratic permutads}\label{BinaryPerm}  

For binary permutads it will prove helpful to replace the permutations by the leveled planar binary trees (lpb trees for short) in order to handle explicitly the operations. We refer to  Appendix 1 for details on this bijection. We show that the binary permutads can be presented by taking only the $\circ_{i}$ operations. As a consequence a binary permutad is equivalent to a binary pre-permutad. 
We will see that it is the relevant tool to study products for which the two ways of computing $((xy)(zt))$ differ. 

\subsection{Definition}
A \emph{binary permutad} is a permutad which is generated by binary operations. In other words, it is the quotient of a free permutad $\PPP(M)$, where the $\NN^{+}$-module $M$ is concentrated in arity 2: $M=(0,E,0,\ldots )$.

By \ref{monadPerm} a typical element of $\PPP(M)$ is a surjection such that the arity of each vertex is $2$ (hence it is a permutation), whose vertices are decorated by  elements of $E$. Under the isomorphism between permutations and lpb trees, cf.\ \ref{trees}, it is given by a lpb tree whose vertices are decorated by elements of $E$.  In \ref{thmpartial} we presented the notion of permutad by means of the operations $\circ_{t}$, for $t$ a surjective map with two vertices, and some relations. We will see that, when we restrict ourself to binary permutads, then the operations $\circ_{i}$ are sufficient to present the notion of binary permutad.

\begin{thm}\label{thmbinarypartial} A binary permutad $\PP$ is completely determined by the arity graded module $\{\PP_{n}\}_{n\geq 1}$, with $\PP_{1}=\KK$, and the partial operations 
$$\circ_{i}: \PP_{m+1}\t \PP_{n+1}\to \PP_{m+n+1},\quad 1\leq i\leq m+1,$$
satisfying the sequential composition relation. As a consequence a binary permutad is equivalent to a binary pre-permutad. 
\end{thm}

\begin{proo}  It suffices to show the statement of the theorem for the free permutad $\mathrm{perm}(Y)$ on one generator, namely $Y:=\arbreA$. By Proposition \ref{propFreePermutad} and the bijection between permutations and  lpb trees, this free permutad is spanned by the lpb trees. 
Hence since $\mathrm{perm}(Y)$ is generated by $Y$ under composition, it suffices to show that any composition of copies of $Y$ is equivalent to a lpb tree under the sequential composition relation. Observe that this is a set-theoretic question. In arity one we have only $\id$ (the tree $|\ $), and in arity two the generator $Y$. In arity 3 there are two ways of composing and no relation yet, so we get $\arbreAB$ and $\arbreBA$. In arity 4 we get 10 possible compositions (2 for each one of the 5 trees) and we get 4 relations:
\begin{gather*}
(-\circ_{1}-)\circ_{1}- = -\circ_{1}(-\circ_{1}-)\ ,\\
(-\circ_{1}-)\circ_{2}- = -\circ_{1}(-\circ_{2}-)\ ,\\
(-\circ_{2}-)\circ_{2}- = -\circ_{2}(-\circ_{1}-)\ ,\\
(-\circ_{2}-)\circ_{3}- = -\circ_{2}(-\circ_{2}-)\ .
\end{gather*}
The quotient is therefore made of 6 elements, one for each of the 4 cases above, and the 2 compositions
$$ (-\circ_{2}-)\circ_{1}- \quad \mathrm{and} \quad ( -\circ_{1}-)\circ_{3}-\ .$$
It is clear that each case corresponds to one of the lpb trees: 
$$\arbreABC \arbreBAC \arbreCAB \arbreCBA \arbreACB \arbreBCA\ .$$
By induction we suppose that we get lpb trees in arity $n-1$ and we are going to prove the same statement in arity $n$. An element is the class of some composite $\omega ' \circ_{k} \omega$ of elements of lower arity. By induction $\omega$ is a lpb tree, hence is of the form $\omega''\circ_{j}Y$. By the sequential composition relation we have
$$\omega ' \circ_{k} (\omega''\circ_{j}Y)=( \omega ' \circ_{k} \omega'')\circ_{k-1+j}Y.$$
Hence by induction this element can be identified with a lpb tree, and we are done.
\end{proo}

\subsection{Algebras over a binary permutad}\label{algbinarypermutad} By definition an \emph{algebra over a binary permutad} $\PP$ is a space $A$ equipped with linear maps 
$$ \PP_{n}\t A^{\t n} \to A,\quad  (\mu; a_{1}\cdots a_{n})\mapsto \mu(a_{1}\cdots a_{n})$$
 for any $n\geq 1$ such that $\id(a)= a$ and 
$$(\mu \circ_{i}\nu )(a_{1}\cdots a_{m+n-1})= \mu(a_{1}\cdots a_{i-1} \nu(a_{i}\cdots a_{i+n-1}) a_{i+n}\cdots a_{m+n-1})\ ,$$
for any $\mu\in \PP_{m}$ and any $\nu\in \PP_{n}$.

\subsection{Binary quadratic permutad} Let $M$ be an arity graded space of generating operations. 
 The space spanned by composition of two operations in $M$ is denoted by $\PPP(M)^{(2)}$.
 
A \emph{quadratic permutad} is a permutad $\PP= \PP(M,R)$ which is presented by generators and relations $(M, R)$ and whose space of relations $R$  is in $\PPP(M)^{(2)}$.

In the binary case (i.e. $M$ is completely determined by its component in arity $2$ denoted $E$) $\PPP(M)^{(2)}$ is the direct sum of two copies of $E\t E$, one for the identity permutation $1_2$ and the other one for the transposition $(12)$:
$$\arbreABdecorated\quad , \quad \arbreBAdecorated\ .$$ 
So the data to present a binary quadratic permutad, resp. pre-permutad, resp.\ ns operad is the same. By Theorem \ref{thmbinarypartial} the permutad and the pre-permutad are the same. The ns operad is a quotient obtained by moding out by the parallel composition relations (see for instance \cite{LV} Chapter 5). 

One can find many examples of binary quadratic ns operads in the Encyclopedia \cite{Zinbiel10} : each one of them gives a permutad. As shown below, in the $q\textrm{-}{\text PermAs}$ case the underlying arity modules can be very different in the operad case and in the permutad case (for instance if $q=-1$ and $n\geq 4$, then $\dim\ (-1)\textrm{-}{\permAs}_{n}=1$ and $\dim\ (-1)\textrm{-}As_{n}=0$).

\subsection{Parametrized associative permutad}\label{parametrizedPermutad} Let $q\in \KK$ be a parameter. We define the \emph{parametrized associative permutad} $q\textrm{-}{\permAs}$ as the permutad generated by one element in arity 2, denoted by $\mu$, and satisfying the relation
$$\Gamma((12);\mu,\mu) = q\, \Gamma(1_{2};\mu,\mu),$$
where $1_{2}$ is the identity and $(12)= [21]$ is the cycle. Equivalently this relation can be written $\mu\circ_2\mu = q\, \mu\circ_1\mu$. 
We will show that $(q\textrm{-}{\permAs})_{n}$ is one-dimensional for any  $n\geq 1$. The permutadic composition gives the following result.

\begin{prop}\label{formulalength} For any $n\geq 1$ the module $(q\textrm{-}{\permAs})_{n}$ is one-dimensional spanned by $\Gamma(1_{n}; \mu, \ldots, \mu)$. For any permutation $\ss\in \Sy_{n}$, considered as a surjective map from $\u{n}$ to $\u{n}$, we have the following equality
$$\Gamma(\ss; \mu, \ldots, \mu) = q^{\ell(\ss)}\Gamma(1_{n}; \mu, \ldots, \mu), $$
where $\ell(\ss)$ is the length of $\ss$ in the Coxeter group $\Sy_{n}$.
\end{prop}
\begin{proo} We consider the Coxeter presentation of the symmetric group $\Sy_{n}$ with generators $s_{1}, \ldots, s_{n-1}$ (transpositions). The length of $\ss$, denoted by $\ell(\ss)$, is the number of Coxeter generators in a minimal writing of $\ss$. Consider the poset of permutations equipped with the weak Bruhat order. So, $\ss\to \tau$ is a covering relation iff $\tau$ is obtained from $\ss$ by the left multiplication by a Coxeter generator and $\ell(\tau)=\ell(\ss)+1$. Under the bijection between permutations and leveled binary trees, cf.\ \ref{trees}, there are two different covering relations:

-- those which are obtained through a local application of 
$$\arbreAB\to \arbreBA,$$

-- those which are obtained by some change of levels, like 
$$\arbreACB\to\arbreBCA .$$

The property that we use is the following, proved in the appendix, cf.\ Proposition \ref{propconnected}:

\emph{the subposet of the poset $\Sy_{n}$ made of the covering relations of first kind is a connected graph.}

The relation which defines the permutad $q\textrm{-}{\permAs}$ is precisely

$$\Gamma(\arbreBA;\mu,\mu) =q \Gamma(\arbreAB;\mu,\mu).$$
Since any $\ss$ can be related to $1_{n}$ by a sequence of covering relations of the first kind, we can apply repeatedly the relation and we get the expected formula.
\end{proo}

We remark that in the permutadic case we do not encounter the obstruction $q^{3}=q^{2}$ met in the operadic case. In the operadic case the module $(q\textrm{-}{\permAs})_{n}$ is one-dimensional, when $n\geq 4$, only for $q=0,1$ or $\infty$ (compare with \cite{MarklRemm09}).
 
 \begin{cor} The associative permutad $\permAs$ presented by a binary operation and the associativity relation is one dimensional in each arity. Hence it is the permutad associated to the ns operad $As$.
 \end{cor}
\begin{proo}  It is a particular case of Proposition \ref{formulalength} since ${\permAs}= 1\textrm{-}{\permAs}$. \end{proo} 
 
 \subsection{Examples of computation} If the permutad $\PP$ has  generating operations $M$, then it is a quotient of the free permutad  $\PPP(M)$. An element in the free permutad $\PPP(M)$ is a ``computational pattern" in the sense that, given a sequence of elements in a $\PP$-algebra $A$, we can compute an element of $A$ out of this pattern. So, the parenthesizings of the operad framework, i.e. planar trees, are replaced here by leveled planar binary trees. For instance, supposing that $M$ is determined by one binary operation,  the two computational patterns
$$\arbreACB \quad \textrm{and} \quad \arbreBCA$$
give, a priori, two distinct values denoted by:
$$((xy)_1(zt)_2) \quad \textrm{and} \quad ((xy)_2(zt)_1)$$
respectively. Here is an explicit example.

In the case of an algebra $A$ over the permutad $q\textrm{-} {\permAs}$ we have
$$((ab)_{2}(cd)_{1})= q ((ab)_{1}(cd)_{2})$$
for any $a,b,c,d\in A$. In the particular case of the free algebra on one generator $x$, the elements $x^n$, defined inductively as $x^n=x^{n-1} x$, span this algebra. We compute
$$((xx)_{1}(xx)_{2})= q x^4 \quad \textrm{and} \quad  ((xx)_{2}(xx)_{1})=q^2 x^4.$$

The study of composition of operations for which one takes care of the order in which these operations are performed has already been addressed in programming theory, see for instance \cite{Curien} and the references therein.


\section{Associative permutad and the permutohedron}\label{s:AssoPermPermutohedron}

In the operad framework the operad $As$, which encodes the associative algebras, admits a minimal model which is described explicitly in terms of the Stasheff polytope (associahedron). It means that this minimal model is a differential graded operad whose space of $n$-ary operations is the chain complex of the $(n-2)$-dimensional \emph{associahedron} (considered as a cell complex). When we consider $As$ as a permutad, denoted by ${\permAs}$, then one can also construct its minimal model. We show that, in arity $n$, this differential graded permutad is given by the chain complex of the $(n-2)$-dimensional \emph{permutohedron}.

\subsection{Associative permutad} We consider an arity-graded module which is $0$ except in arity $2$ for which it is one-dimensional, spanned by $\mu$. The free permutad on $\mu$, denoted $\permMag(\mu)$ is such that $\permMag(\mu)_{n}\cong \Sy_{n-1}$ since the non-zero decorated surjective maps are bijections, cf.\ Proposition \ref{propFreePermutad}. For instance, in arity $3$, we get

\raisebox{-0.7cm}{$\mu\circ_{1}\mu =\Gamma(1_{2};\mu,\mu)=\qquad $} {$\xymatrix{ \bullet\ar[d] &  \bullet\ar[d] \\ \times\atop \mu &  \times\atop \mu }$\raisebox{-0.7cm}{\qquad =\qquad      $[1 \  2]$}

\raisebox{-0.7cm}{$\mu\circ_{2}\mu =\Gamma((12);\mu,\mu)=\qquad $} {$\xymatrix{ \bullet\ar[dr] &  \bullet\ar[dl] \\ \times\atop \mu &  \times\atop \mu }$\raisebox{-0.7cm}{\qquad =\qquad      $[2\ 1]$}

Let us put the relation $\Gamma(1_{2};\mu,\mu)= \Gamma((12);\mu,\mu)$ (which is the associativity relation $\mu\circ_{1}\mu =\mu\circ_{2}\mu$) and denote the associated permutad by ${\permAs}$ (this is the permutad $1\textrm{-}{\permAs}$ introduced in \ref{parametrizedPermutad}). 

\subsection{A quasi-free dg permutad}  We construct a  quasi-free dg permutad ${\permAs}_{\infty}$ as follows. Let $V$ be the arity-graded module which is one-dimensional in each arity $n\geq 2$ and $0$ in arity $1$. We denote the linear generator in arity $n$ by $m_{n}$, so $V_{n}= \KK\ m_{n}$. We declare that the homological degree of $m_{n}$ is $n-2$. As a permutad  ${\permAs}_{\infty}$ is the free permutad on the graded $\NN^+$-module $V$ (two gradings: homological and arity). It comes immediately from Proposition \ref{propFreePermutad} that $({\permAs}_{\infty})_{n}$ can be identified to the free module on the surjections (i.e. the set $\Surj_{n}$), hence the cells of the permutohedron of dimension $n-2$, cf.\ \ref{Pchaincplx}. Under this identification the operation $m_{n}$ is identified  with the big cell $c_n$ of the $(n-2)$-dimensional permutohedron. We put on it the boundary map of the permutohedron, cf.\ \ref{Pchaincplx}. There is a unique extension of the bounday map to the free permutad by universality of a free object.
 
 \begin{thm} The permutadic structure of ${\permAs}_{\infty}$ is compatible with the boundary map, hence ${\permAs}_{\infty}$ is  a dg quasi-free permutad such that 
 $$({\permAs}_{\infty})_{n} \cong C_{\bullet}({\texttt P}^{n-2}).$$
  \end{thm}
 \begin{proo} We introduced the notion of dg permutad in \ref{dgpermutad}. We need to prove that the structure map 
 $$\PPP(\PPP(V))_{n}\to \PPP(V)_{n}$$
 is compatible with the boundary map. It is sufficient to check that the substitution at any vertex is compatible, that is, to check that for $t:\u{m+n} \to \u{2}$  the map
 $$\circ_{t} : \PPP(V)_{m+1}\t \PPP(V)_{n+1}\to   \PPP(V)_{m+n+1}$$
 commutes with the differential. Since any cell of $\texttt{P}^k$ is a product of permutohedrons of lower dimensions, it suffices to check this property for $c_{n}\in \PPP(V)_{n}$ and  $c_{m}\in \PPP(V)_{m}$. The element $c_{m}\circ_{t}c_{n}$ is in fact a cell of the permutohedron which is the product of two permutohedrons $\texttt{P}^{m-2}\times \texttt{P}^{n-2}$. Its boundary is 
 $$\partial(\texttt{P}^{m-2}\times \texttt{P}^{n-2})= \partial(\texttt{P}^{m-2})\times \texttt{P}^{n-2} \cup\texttt{P}^{m-2}\times \partial(\texttt{P}^{n-2}).$$
  Therefore we have, at the chain complex level, 
 $$d(c_{m}\circ_{t}c_{n}) = dc_{m}\circ_{t}c_{n} \pm c_{m}\circ_{t}dc_{n}$$
 as expected. \end{proo}
 
 \begin{prop} The dg permutad ${\permAs}_{\infty}$ is a quasi-free model of the permutad ${\permAs}$.
   \end{prop}
 \begin{proo}
The augmentation map ${\permAs}_{\infty}\to {\permAs}$ sends all the $0$-cells to $\mu_{n}$ and the other higher dimensional cells to $0$. It is obviously a map of dg permutads (trivial differential graded structure for ${\permAs}$). It is a quasi-isomorphism, since the permutohedron is contractible. So we have constructed a quasi-free model of the permutad ${\permAs}$. 
  \end{proo}
 
\subsection{${\permAs}$-algebras up to homotopy} From the explicit description of the minimal model of the permutad ${\permAs}$ we get the following definition of a ${\permAs}_{\infty}$-algebra, analogous to the definition of an $A_{\infty}$-algebra.

A ${\permAs}_{\infty}$-algebra is a chain complex $(A,d)$ over $\KK$ equipped with linear maps of degree $n-2$:
$m_t: A^{\t n} \to A$
for any cell $t$ of the permutohedron $\texttt{P}^{n-2}$. These maps are supposed to satisfy the following properties: 
$$\partial(m_t) = \sum_s \pm m_s,$$
where the sum is over the cells $s$ of codimension 1 in the boundary of the cell $t$.

\noindent {\sc Examples.} Let us adopt the shuffle notation for the cells of the permutohedron. So the top cell of $\texttt{P}^{n-2}$ is $\{1\ 2 \ \ldots n-1\}$ and the map $m_{\{1\ 2 \ \ldots\ n-1\}}$ is playing the role of the operadic operation $m_n$.  In low dimensions the formulas are the following:

\begin{alignat*}
\partial(m_{\{1 \}}) &= 0\\
\partial(m_{\{1 2 \}}) &=  m_{\{1\ | \ 2 \}} -  m_{\{2\ | \ 1 \}}\\
\partial(m_{\{1 2  3 \}}) &= m_{\{12\ | \ 3 \}} + m_{\{2\ |\ 1 3 \}}+ m_{\{2 3\ | \ 1 \}}- m_{\{1\ |\ 2 3 \}}- m_{\{1 3\ | \ 2 \}}- m_{\{3\ |\ 1 2 \}}.
\end{alignat*}

The interest of this structure lies in the following ``Homotopy Transfer Theorem": if a chain complex $(W,d)$ is  an algebra over the permutad ${\permAs}$, then any deformation retract $(V,d)$ of $(W,d)$ is naturally equipped with a structure of ${\permAs}_{\infty}$-algebra. This is part of a Koszul duality theory for permutads, which will be worked out elsewhere.


\section{The permutad of associative algebras with derivation}

We describe explicitly the permutad of associative algebras equipped with a derivation. We show that, in arity $n$, it is the algebra of noncommutative polynomials in $n$ variables. Recall that, in the operad framework, it is the commutative polynomials which occur, cf.\ \cite{AsDer}.

\subsection{Permutads with $1$-ary operations} In the previous sections we supposed, for simplicity, that a permutad had only one $1$-ary operation, namely $\id$.  But there is no obstruction to extend this notion. When working with the leveled planar trees, for instance, it suffices to admit the vertices with one input (for instance the ladders). We need this generalization for the following  example.

\subsection{The permutad ${\permAsDer}$} We denote by ${\permAsDer}$ the permutad which is generated by a unary operation $D$ and a binary operation $\mu$, which satisfy the following relations:
\begin{displaymath}
\left\{
\begin{array}{rcl}
\mu \circ_{1} \mu & = & \mu \circ_{2} \mu\ ,\\
D \circ_{1}\mu & = & \mu \circ_{1}D + \mu \circ_{2} D\ ,\\
(\alpha\circ _{i}D)\circ _{j}\mu &=& (\alpha\circ _{j}\mu)\circ _{i}D\ ,\\
(\alpha\circ _{i}\mu)\circ _{j+1}D &=& (\alpha\circ _{j}D)\circ _{i}\mu\ ,
\end{array} \right.
\end{displaymath}
for any operation $\alpha$ and $i<j$.

Observe that the first relation is the associativity of $\mu$, the second relation is saying that $D$ is a derivation, the third and fourth relations say that the operations $D$ and $\mu$ commute for parallel composition.

\begin{thm}\label{permAsDerexplicit} As a vector space ${\permAsDer}_{n}$ is isomorphic to the space of noncommutative polynomials in $n$ variables:
$${\permAsDer}_{n} = \KK\langle x_{1}, \ldots , x_{n}\rangle .$$
Let $t:\u{n} \to \u{2}$ be the surjective map given by the unshuffle $\{i_1, \ldots ,i_m\ |\ j_1 , \ldots , j_{n}\}$. The composition map $\circ_{t}$ is given by
\displaylines{
(P\circ_{t}Q)(x_{1}, \ldots , x_{n+m-1})= \hfill \cr
\qquad \qquad \qquad P(x_{j_1}, \ldots, x_{j_{i-1}}, x_{i_1}+ \cdots + x_{i_m},x_{j_i},  \ldots , x_{j_{n}})Q(x_{i_1}, \ldots , x_{i_m}).}

Under this identification the operations $\id, D, \mu$ correspond to $1_{1}, x_{1}\in \KK\langle x_{1} \rangle$ and to $1_{2}\in \KK\langle x_{1}, x_{2}\rangle$ respectively. More generally the operation 
$$\big(\cdot ((\mu\circ_{j_{k}}D)\circ_{j_{k-1}}D)\cdots \circ_{j_{1}}D\big)$$
 corresponds to the noncommutative monomial $x_{j_{k}}x_{j_{k-1}}\cdots x_{j_{1}}$.
\end{thm}
Graphically the operation $x_{j_{k}}x_{j_{k-1}}\cdots x_{j_{1}}$ is pictured as a  planar decorated tree with levels as follows (example: $x_{1}x_{2}x_{n}x_{2}$) :
$$\xymatrix@R=4pt@C=4pt{
\ar@{-}[dddd] & \ar@{-}[d] & \cdots &  \ar@{-}[dd] \\
 & D\ar@{-}[dd]  & \cdots & \\
 &   & \cdots & D \ar@{-}[dd] \\
 & D\ar@{-}[d]  &  \cdots & \\
D\ar@{-}[drr]  &*{} \ar@{-}[dr]  & \cdots &*{} \ar@{-}[dl]  \\
                    &                   &*{} \ar@{-}[d]    &   \\
                                &                   &*{}      &     }
$$

\begin{proo} Up to a minor change of notation and terminology this result has been proved in \cite{AsDer} for the case $\circ_t = \circ_i$, that is when the inverse image of $1$ for $t$ is made of consecutive elements. The proof for any $t$ is similar.
\end{proo}


\section{Permutads and shuffle operads}\label{s:ShuffleOp}

Following the work of E.\ Hoffbeck \cite{Hoffbeck10}, V.\ Dotsenko and A.\ Khoroshkin introduced in \cite{DotsenkoKhoroshkin09} the notion of shuffle operad. It is based on  the combinatorial objects with substitution called shuffle trees. It turns out that the surjective maps can be considered as particular shuffle trees, whence the relationship between shuffle operads and permutads.

\subsection{Shuffle trees and shuffle operads}\label{shuffleTreeOp} A \emph{shuffle tree} is a planar tree equipped with a labeling of the leaves by integers $\{0, 1 , 2 ,  \ldots , n \}$ satisfying some condition stated below. First, we label each edge of the tree as follows. The leaves are already labelled. Any other edge is the output of some vertex $v$ of the tree. We label this edge by $\mathrm{min}(v)$ which is the minimum of the labels of the inputs of $v$. Second, the condition for a labeled tree to be called a shuffle tree is that, for each vertex, the labels of the inputs, read from left to right, are increasing.

\textsc{Example:} 

$$\ShuffleTree$$

The substitution of a shuffle tree at a vertex of a shuffle tree still gives a shuffle tree. Hence we can define a monad on arity-graded modules (see \cite{DotsenkoKhoroshkin09}, or \cite{LV}, Chapter 8 for details). An algebra over this monad is by definition a \emph{shuffle operad}.

\subsection{Left combs and permutads}\label{leftcombs} A \emph{left comb} is a planar tree such that, at each vertex, all the inputs but possibly the most left one, is a leaf. A \emph{shuffle left comb} is a shuffle tree whose underlying tree is a left comb. There is a bijection between shuffle left combs and  surjective maps as follows. Order the vertices of the left comb downwards from $1$ to $k$. The surjective map $f:\u{n}\to \u{k}$ is such that $f^{-1}(j)$ is the set of labels of the leaves pertaining to the vertex number $j$. Here is an example:

$$\xymatrix{
0 & 1 & 3 & 4 & 2 & 5 \\
 & *{}\ar@{-}[ul]\ar@{-}[u] \ar@{-}[ur] \ar@{-}[urr]&&&&\\
&&&& *{}\ar@{-}[ulll] \ar@{-}[uu] \ar@{-}[uur] \ar@{-}[d]&\\
&&&&*{} &\\
}$$
gives the surjective map:
$$\xymatrix{
  *{\bullet}\ar@{-}[d] & *{\bullet}\ar@{-}[drr] & *{\bullet}\ar@{-}[dll] & *{\bullet}\ar@{-}[dlll] & *{\bullet}\ar@{-}[dl]\\
  *{\times}& &&*{\times}& &      \\
  }$$
  
\begin{prop}\label{prop:PermtoShuffleOp} The bijection between surjections and shuffle left combs is compatible with substitution. As a consequence any shuffle operad gives rise to a permutad. For instance any symmetric operad gives rise to a permutad.
\end{prop}
\begin{proo}  The first statement is proved by direct inspection. Since a shuffle operad is an algebra over the monad of shuffle trees, it suffices to restrict oneself to the shuffle left combs to get a permutad.

 It is shown in \cite{DotsenkoKhoroshkin09} that any symmetric operad $\PP=\{\PP(n)\}_{n\geq 1}$ gives a shuffle operad $\PP^{\text sh}=\{\PP^{\text sh}_{n}\}_{n\geq 1}$ with $\PP^{\text sh}_{n}= \PP(n)$. Hence a fortiori any symmetric operad gives rise to a permutad.
\end{proo}

\subsection{On the ``partial'' presentation of a permutad} Replacing surjections by left comb shuffle trees we observe that any left comb shuffle tree can be obtained by successive substitutions of left comb shuffle trees which have only 2 vertices.  In fact these trees are the only left comb shuffle trees with 2 levels and also the only ones which give partial operations (compare with section 8.2 of \cite{LV}).

\subsection{Algebras over a permutad} We can define an algebra $A$ over any permutad $\PP$ as a morphism of permutads $\PP\to \End(A)$ where the permutad structure of $\End(A)$ comes from its structure of symmetric operad, hence shuffle operad, hence permutad by restriction. But observe that, in the case of binary permutads, it is  different from the notion of algebra defined in \ref{algbinarypermutad} since here it involves the action of the symmetric group.

\subsection{The permutad associated to the symmetric operad $Ass$} We consider the symmetric operad $Ass$ encoding the associative algebras. We know that $\PP(n)=\KK[\Sy_n]$. Hence the shuffle operad $Ass^{sh}$ associated to it is such that $Ass^{sh}_n=\KK[\Sy_n]$. Let us denote by ${\permAs}^{sh}$ this shuffle operad viewed as a permutad. It has two linear generators in arity $2$ that we denote by $1$ (the identity in the group $\Sy_2$) and $\tau$ (the flip in $\Sy_2$) respectively. These two operations generate $8$ operations in arity $3$. Since $Ass^{sh}_3=\KK[\Sy_3]$ is of dimension $6$, it means that there are two quadratic relations. An easy computation shows that they are:

$$\xymatrix{
*{\bullet}\ar@{-}[d] & *{\bullet}\ar@{-}[d]  && *{\bullet}\ar@{-}[rd] & *{\bullet}\ar@{-}[ld]  \\
*{\times\atop  1}       & *{\times\atop   \tau}                &=& *{\times\atop   \tau}   & *{\times\atop   1}          
}$$
\vskip 0.5cm
$$\xymatrix{
*{\bullet}\ar@{-}[d] & *{\bullet}\ar@{-}[d]  && *{\bullet}\ar@{-}[rd] & *{\bullet}\ar@{-}[ld]  \\
*{\times\atop   \tau}        & *{\times\atop   1}   &=& *{\times\atop   1}    &  *{\times\atop   \tau}          
}$$

This permutad looks analogous to the ns operad ${\text Dend}$ encoding dendriform algebras. Indeed, for each of them the dimension of the space of operations is the same as the dimension of the free object on one generator, shifted by one.


\section{Appendix 1: Combinatorics of  surjections on finite sets}\label{appendixsurj} 

In this paper we are using four different combinatorial ways of encoding the same objects: shuffle left combs, shuffles, surjections, leveled planar trees. In this appendix we describe explicitly three of the bijections between these families of combinatorial objects.

$$\begin{matrix}
\textrm{shuffle left comb} & \textrm{shuffle} & \textrm{surjection} & \textrm{leveled planar tree} \\
&&&\\
n+1 \textrm{ leaves, } k \textrm{ vertices} & \ss\in \Sy_n,\ k \textrm{ subsets} & t:\u{n}\to \u{k} & n+1 \textrm{ leaves, } k \textrm{ levels} \\
&&&\\
\arbreAsh & [1]  &
\xymatrix@R=1pt@C=1pt{
\bullet \ar@{-}[dd]\\
&\\
  *{\times} \\} 
   &\arbreA\\
\arbreABsh & [1 | 2 ] &
\xymatrix@R=1pt@C=1pt{
\bullet \ar@{-}[dd]& & \bullet  \ar@{-}[dd]\\
&&\\
  *{\times} & & *{\times} \\} 
   & \arbreAB\\
\arbreABshdeux & [2 | 1 ] &
\xymatrix@R=1pt@C=1pt{
\bullet \ar@{-}[rrdd]& & \bullet  \ar@{-}[lldd]\\
&&\\
  *{\times} & & *{\times} \\} 
   & \arbreBA\\
\arbreABCsh & [1 | 2 | 3] &
\xymatrix@R=1pt@C=1pt{
\bullet \ar@{-}[dd]& & \bullet  \ar@{-}[dd]& &\bullet \ar@{-}[dd]\\
&&&&\\
  *{\times} & & *{\times} && *{\times}\\} 
   & \arbreABC\\
\arbreAACsh & [1 2 | 3] &
\xymatrix@R=1pt@C=1pt{
\bullet \ar@{-}[rdd]& & \bullet  \ar@{-}[ldd]& &\bullet \ar@{-}[dd]\\
&&&&\\
 & *{\times} & && *{\times} \\} 
  & \arbreAAC\\
\arbreABCshdeux & [1 | 3 | 2 ] &
\xymatrix@R=1pt@C=1pt{
\bullet \ar@{-}[dd]& & \bullet  \ar@{-}[rrdd]& &\bullet \ar@{-}[lldd]\\
&&&&\\
  *{\times} & & *{\times} && *{\times}\\} 
  & \arbreACB\\
\arbreAACshdeux& [1 3 | 2 ] &
\xymatrix@R=1pt@C=1pt{
\bullet \ar@{-}[rdd]& & \bullet  \ar@{-}[rdd]& &\bullet \ar@{-}[llldd]\\
&&&&\\
 & *{\times} & & *{\times} & \\} 
 & \arbreACA
\end{matrix}$$

A shuffle left comb is a left comb whose leaves have been enumerated so that, at each vertex, the numbers of the leaves are increasing from left to right. The bijection from shuffles to shuffle left combs is given as follows. Let $\ss =[l_1^1,\dots ,l_{i_1}^1 | l_1^2,\dots | l_1^{k},\dots ,l_{i_k}^k]$ be a $k$-shuffle. The number of inputs of the upper vertex of the left comb is $1+i_1$, of the next one it is $1+i_2$, etc. The decoration of the leaves are 
$0 , l_1^1,\dots ,l_{i_1}^1,l_1^2,\dots ,l_1^{k},\dots ,l_{i_k}^k$.

The bijection from shuffles to surjections is given as follows. Let $\ss$ be a shuffle as above. The surjective map $t:\u{n}\to \u{k}$ is determined by
$t^{-1}(j) =\{  l_1^{j},\dots ,l_{i_j}^j\}$.

 We consider the set of leveled planar rooted trees with $n+1$ leaves.  By ``leveled" we mean that any vertex is assigned a level. Of course the levels of the vertices are compatible with the structure of the tree. For instance the following trees are admissible and different:
$$\arbreACB \quad  \arbreBCA$$
The bijection from the leveled trees to the surjections is given as follows. First, we label the leaves from left to right by $0, 1, \ldots , n$. Second, we label the levels downwards from $1$ to $k$. Let $t$ be the searched surjection. The integer $t(i)$ is the number of the level which is attained by a ball which is dropped in between the leaves $i$ and $i+1$.


\section{Appendix 2: the permutohedron}\label{appendixassoc} We define and construct the permutohedron together with several ways of labelling its cells: either by  surjections or by leveled planar rooted trees. Then, we prove a lemma on some subposet of the weak Bruhat poset of permutations.

\subsection{The permutohedron as a cell complex}\label{cellcplx} For any permutation  $\ss\in \Sy_{n}$ we associate a point $M(\ss)\in \RR^n$ with coordinates
$$M(\ss) =(\ss(1), \ldots , \ss(n)).$$
Since $\sum_{i}\ss(i)= \frac{n(n+1)}{2}$, the points $M(\ss)$ lie in the hyperplane $\sum_{i}x_{i}= \frac{n(n+1)}{2}$ of $ \RR^n$. The convex hull of the points $M(\ss), \ss\in \Sy_{n}$, forms a convex polytope $\texttt{P}^{n-1}$ of dimension $n-1$, whose vertices are precisely the $M(\ss)$'s.
$$\begin{array}{lccccccc}
\texttt{P}^n&\Kzero &&  \KunA &&  \PdeuxA &&  \PtroisA \\
&&&&&&&\\
&&&&&&&\\
n=& 0 && 1 && 2 && 3\\
\end{array}$$

The polytope $\texttt{P}^{n}$ is called the \emph{permutohedron}.  It is a cell complex. The cells of the permutohedron $\texttt{P}^{n-1}$ are in one to one correspondence with the surjective maps $\u{n} \epi \u{k}$, cf.\ for instance \cite{PalaciosRonco}. For $k=n$ we obtain the bijection between the set of vertices and the permutations since any surjective map between finite sets is bijective. For $k=1$ there is only one map which corresponds to the big cell. More generally a surjective map $t=\u{n} \epi \u{k}$ corresponds to a $n-k$ dimensional cell $\texttt{P}^{n-1}$. Let $i_{j}= \#  t^{-1}(j)$. The subcell corresponding to $t$ is of the form $\texttt{P}^{i_{1}-1}\times \cdots \times \texttt{P}^{i_{k}-1}$.

\subsection{Examples}\label{list} The permutohedron in dimension $1$ and $2$ together with their associated surjective maps:

$$\xymatrix@R=16pt@C=1pt{
 *{} \ar@{-}[d] & & *{} \ar@{-}[d] & &&& &*{} \ar@{-}[dr] & &*{} \ar@{-}[dl] &  &&& &*{} \ar@{-}[drr] & & *{} \ar@{-}[dll] \\
*{} & *{} & *{} & *{} & *{} & *{} & *{} & *{} & *{} & *{} &  *{}& *{} & *{}& *{} &  *{}& *{}& *{} & *{}\\
 & \bullet\ar[rrrrrrrrrrrrrr] &&&&&&&&&&&&&&\bullet &\\
 }$$

In the following picture the surjections with three, resp.\ two, resp.\ one, outputs encode the vertices, resp.\ edges, resp.\ 2-cell of $\texttt{P}^{2}$:
\begin{figure}[h]
\centering
\includegraphics[scale=0.3]{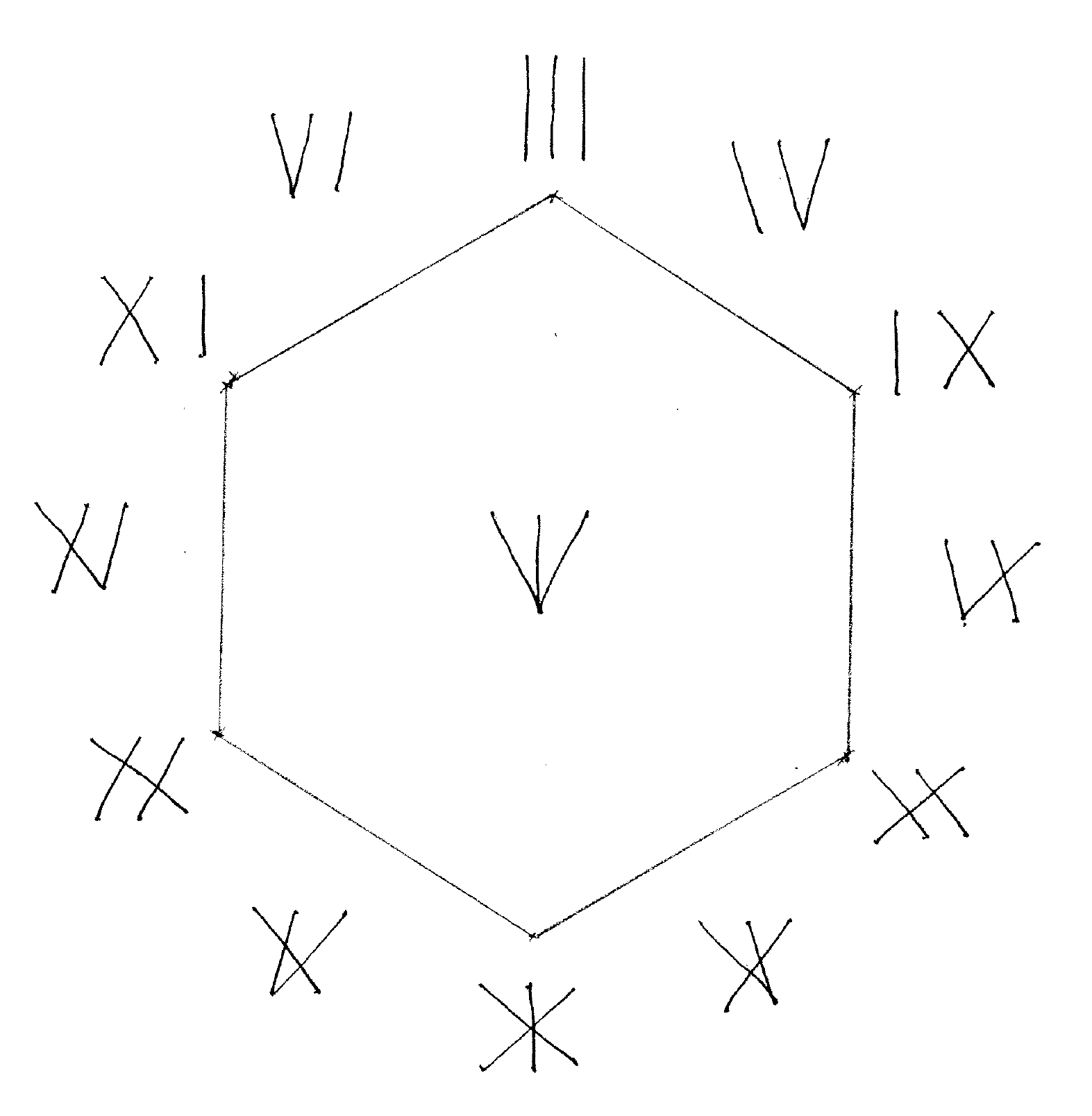}
\caption{The permutad clock} \label{clock}
\end{figure}

In figure \ref{figuregraphpermutations} below we indicate the bijections labelling the vertices.  

\subsection{The chain complex of the permutohedron}\label{Pchaincplx} Since the permutohedron is a cell complex, we can take its associated chain complex $C_{\bullet}(\texttt{P}^{n})$. In degree $k$, the space 
$C_{k}(\texttt{P}^{n})$ is spanned by the $k$-dimensional cells, that is by the surjective maps $\u{n+1} \to \u{n+1-k}$. The boundary map on the big cell $c_{n}$ is given by the formula
$$ d(c_{n})= \sum_{t} \sgn(t) t  $$
where the sum is over all the surjective maps $t$ with target $\u{2}$. The sign $\sgn(t)$ involved in this formula is $\sgn(t):= \sgn(\ss _t)(-1)^{\# t^{-1}(1)}$ where  $\ss _t$ is the shuffle associated to the surjective map $t$ (cf.\ Lemma \ref{lemmasurjshuffle}). See \ref{list} for examples.

\subsection{Vertices of the permutohedron and leveled planar binary trees}\label{trees} The bijection from surjections to leveled planar trees can be restricted to the permutations and the leveled planar binary trees.  Forgetting the levels of the trees gives a surjective map
$$\varphi: \Sy_n \to {\text PBT}_{n+1}.$$
This map appears naturally when comparing the free dendriform algebra on $n$ generators with the free associative algebra, cf.\ \cite{Loday01}. 

\subsection{On a property of the weak Bruhat order}

The 1-skeleton of the permutohedron can be oriented such that each edge 
$$\xymatrix@R=1pt@C=20pt{\ss &\omega\\
\bullet\ar[r] & \bullet} $$
corresponds to the left multiplication by some Coxeter generator $\omega = s_{i}\ss $ such that $\ell(\omega)= \ell(\ss)+1$. The partial order spanned transitively by this covering relation is called the \emph{weak Bruhat order}. So the 1-skeleton is the geometric realization of a poset with minimal element $[1\ 2\ \ldots\ n]=1_{n}$ and maximal element $[n\ n-1\ \ldots \ 1]= s_{1}s_{2}\ldots s_{n-1}\ldots s_{2}s_{1}$.

Under the bijection between permutations and leveled binary trees, there are two different \emph{types of covering relations}:

(1) those obtained through a local application of $\arbreAB\mapsto \arbreBA$, which correspond to $\omega = s_i\ss$ for $\ss $ satisfying that, for all $\ss ^{-1}(i)<j<\ss^{-1}(i+1)$, the integer $\ss (j)<i$,

(2) those obtained by some change of levels, like $\arbreACB\mapsto\arbreBCA$. which correspond to $\omega = s_i\ss$ for $\ss $ satisfying that there exist al least one integer $j$ such that $\ss ^{-1}(i)<j<\ss^{-1}(i+1)$ and $\ss (j)> i+1$.

Edges determined by a covering relation of type (1) are characterized by the following property of their leveled tree: there is only one vertex per level.

In $\texttt{P}^{2}$ there is only one edge which corresponds to a covering relation of type (2). It corresponds to the surjection:

$$\xymatrix@R=8pt@C=8pt{
\bullet \ar@{-}[rdd]& & \bullet  \ar@{-}[rdd]& &\bullet \ar@{-}[llldd]\\
&&&&\\
 & \times & & \times & \\
 }$$
(dotted arrow in Figures \ref{figuregraphpermutations} and \ref{figuregraphtrees}). 
\medskip

The covering relation, from a permutation to another one, is obtained by left multiplication by some Coxeter generator $s_{i}$. The effect consists in exchanging the elements $i$ and $i+1$ in the image of the permutation. If the covering relation implied is of type (1), we call it an \emph{admissible move}. If the elements which lie in the interval between $i$ and $i+1$ contains only elements which are less than $i$, then multiply by $s_{i}$ is admissible. For instance, for $n=3$, the only covering relation which is not admissible is $s_{1}[132]$, see Figure \ref{figuregraphpermutations} and Figure \ref{figuregraphtrees}.

\begin{figure}[h]
\centering
$$\permutohedronLemma$$
\caption{$\texttt{P}^{2}$ and bijections} \label{figuregraphpermutations}
\end{figure}

\bigskip

\begin{figure}[h]
\centering
{\centerline {\scalebox{0.5}{\includegraphics{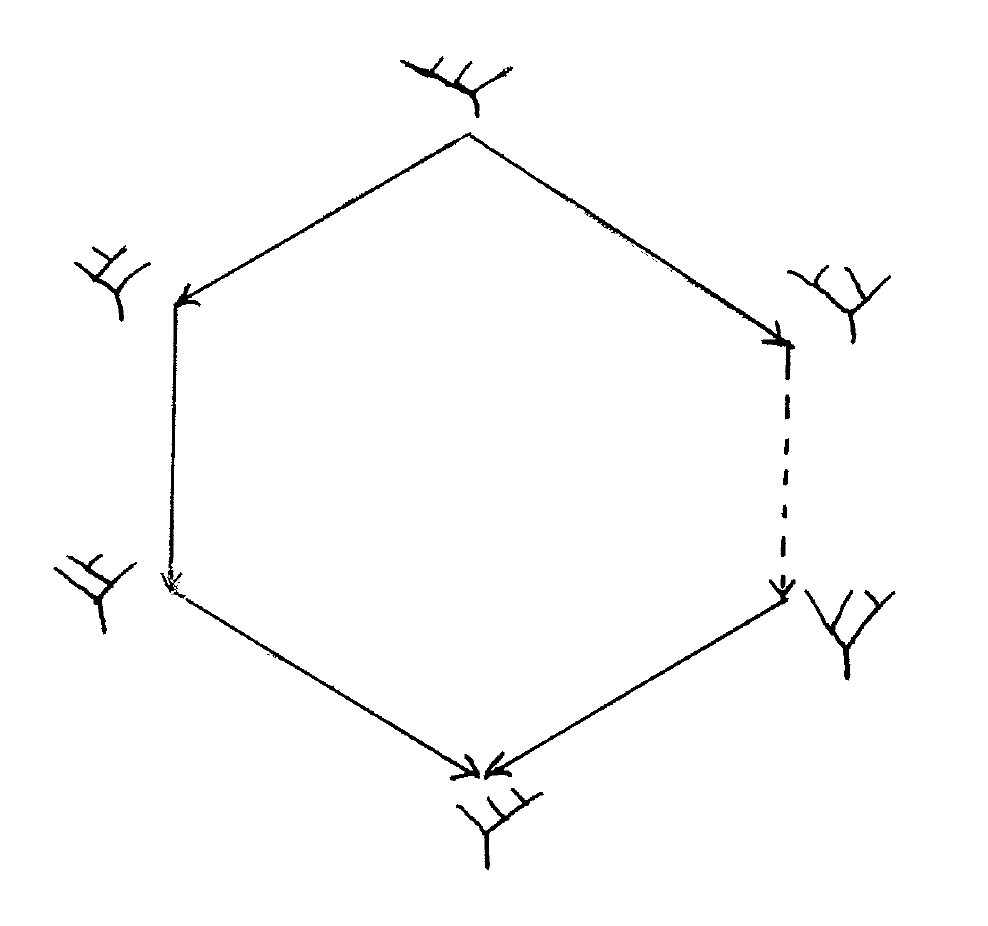}}}}
\caption{$\texttt{P}^{2}$ and trees}\label{figuregraphtrees}
\end{figure}

The following result is used in the proof of Proposition \ref{formulalength} which determines the associative permutad.

\begin{prop}\label{propconnected} The subposet of the weak Bruhat poset $\Sy_{n}$ made of the covering relations of type (1) is a connected graph.
\end{prop}

Let us denote a permutation $\ss$ by its image $[\ss(1)\ldots \ss(n)]$. We first prove a Lemma.

\begin{lemma}\label{lemmatech} Let $\ss\in\Sy_n$ be a permutation such that  $\ss^{-1}(i)<\ss^{-1}(i+2)<\ss^{-1}(i+1)$, for some $1\leq i\leq n-1$, and, 
the unique integer $j$ such that $\ss ^{-1}(i)<j<\ss^{-1}(i+1)$ and $\ss(j)>i+1$ is $\ss^{-1}(i+2)$.
The permutation $s_i\ss$ may be obtained from $\ss $ by admissible moves. \end{lemma}
\begin{proof} It suffices to show that one can exchange $i$ and $i+1$. It follows from the following sequence of moves:
\begin{gather*}
[\ldots i \ldots i+2  \ldots i+1 \ldots]\\
[\ldots i \ldots i+1  \ldots i+2 \ldots]\\
[\ldots i+1 \ldots i  \ldots i+2 \ldots]\\
[\ldots i+2 \ldots i  \ldots i+1 \ldots]\\
[\ldots i+2 \ldots i+1  \ldots i \ldots]\\
[\ldots i+1 \ldots i+2  \ldots i \ldots] \qedhere\end{gather*}
\end{proof}

\begin{proof} [Proof of Proposition \ref{propconnected}] We will show that, starting with some permutation $\ss$, there is always a path of admissible moves leading to the permutation with $i$ and $i+1$ exchanged. We use a double recursive argument on $l=n-i$ and on the number of elements between $i$ and $i+1$ in the image of $\ss$ which are larger that $i+1$.

We need to show that for any permutation $\ss\in S_n$ such that $\ss^{-1}(i)<\ss ^{-1}(i+1)$, for some $1\leq i\leq n-1$, there exists a collection of permutations $\ss_1,\ldots ,\ss_m$ such that :\begin{enumerate}
\item $\ss_1 =\ss$ and $\ss_m=s_i\ss $,
\item for all $1\leq j\leq m-1$ either $\ss _j$ is obtained from $\ss_{j+1}$ by an admissible move, or $\ss _{j+1}$ is obtained from $\ss_j$ by an admissible move.\end{enumerate}

We first proceed by induction on $l=n-i$. If $l=1$, then $i=n-1$ and there exists no $\ss ^{-1}(n-1)<j<\ss ^{-1}(n)$ such that $\ss (j)>n$, so $s_{n-1}\ss$ is obtained from $\ss$ by an admissible move.

For $l>1$, let $\{j_1<\dots <j_r\}$ be the set of integers satisfying that $\ss^{-1}(i)<j_s<\ss^{-1}(i+1)$ and $i+1<\sigma (j_s)$, for $1\leq j\leq r$ .

For $l=2$, we have that $i=n-2$ and $r\in \{0,1\}$. If $r=0$, then $s_{n-2}\ss $is obtained from $\ss$ by an admissible move. If $r=1$, then Lemma \ref{lemmatech} proves the result.

Suppose now that the result is true for all integers $k\leq l$ and for all $0\leq r\leq k-1$. We consider the case $l+1$, with $i=n-l-1$ and $0\leq r\leq l$, and proceed by a recursive argument on $r$.
If $r=0$, then the result is obviously true.

If $r\geq1$, then using repeatedly that the result is true for $l$, we get that $[\ldots i\ldots j_1\ldots j_r\ldots i+1\ldots ]$ is connected by $1$-covering relations to $[\ldots i\ldots i+2\ldots j_2\ldots  j_r\ldots i+1\ldots ]$. Again, applying many timesthat the result is true for $l-1$, we get that $[\ldots i\ldots j_1\ldots j_r\ldots i+1\ldots ]$ is connected by admissible moves to $[\ldots i\ldots i+2\ldots i+3\ldots  j_2\ldots j_r\ldots i+1\ldots ]$. By the same argument, we may replace at each step $j_k$ by $i+k+1$, and finally get that $\ss =[\ldots i\ldots j_1\ldots j_r\ldots i+1\ldots ]$ is connected to $[\ldots i\ldots i+2\ldots i+3\ldots i+r+1\ldots i+1\ldots ]$. So, it suffices to prove that $\ss$ is connected by admissible moves to $[\ldots i+1\ldots i+2\ldots i+3\ldots i+r+1\ldots i\ldots ]$.

We apply now the recursive argument on $r$. If $r=1$, we get the following sequence of permutations which are connected:
\begin{gather*}
[\ldots i \ldots i+2  \ldots i+1 \ldots] \\
[\ldots i \ldots i+1  \ldots i+2 \ldots] \\
[\ldots i+1 \ldots i  \ldots i+2 \ldots] \\
[\ldots i+2 \ldots i  \ldots i+1 \ldots] \\
[\ldots i+2 \ldots i+1  \ldots i \ldots] \\
[\ldots i+1 \ldots i+2  \ldots i \ldots]\end{gather*}
If $r>1$, applying recursive hypothesis on both $l$ and $r$, we get the following sequence of moves:
\begin{gather*}
[\ldots i \ldots i+2  \ldots i+r+1\ldots i+1 \ldots] \\
[\ldots i \ldots i+1  \ldots i+3 \ldots i+r+1\ldots i+2\ldots ] \\
[\ldots i+1 \ldots i  \ldots i+3 \ldots i+r+1\ldots i+2\ldots ]  \\
[\ldots i+2 \ldots i  \ldots i+3 \ldots i+r+1\ldots i+1 \ldots] \\
[\ldots i+2 \ldots i+1  \ldots i+3 \ldots i+r+1\ldots i \ldots] \\
[\ldots i+1 \ldots i+2 \ldots i+3 \ldots i+r+1\ldots i \ldots]\qedhere\end{gather*}
 \end{proof}

Here is the graph corresponding to $\texttt{P}^{3}$:

{\centerline {\scalebox{0.3}{\includegraphics{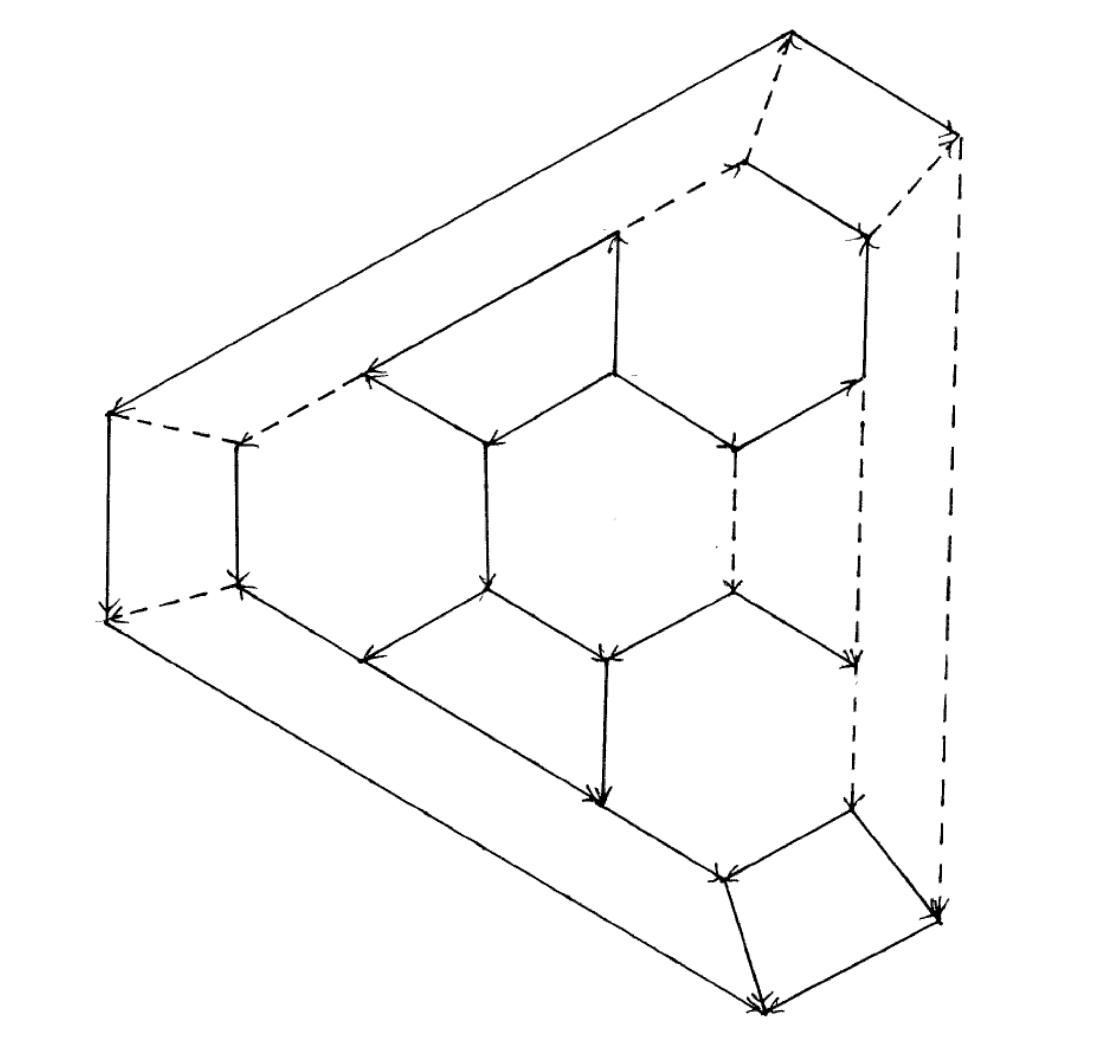}}}}


\end{document}